\documentclass[reqno]{amsart}
\usepackage{hyperref}
\def\N{{\mathbb{N}}}

\def\R{{\mathbb{R}}}
\def\D{{\mathbb{D}}}

\theoremstyle{plain}
\newtheorem{theorem}{Theorem}
\newtheorem{proposition}{Proposition}
\newtheorem{definition}{Definition}
\newtheorem{lemma}{Lemma} 
\newtheorem{corollary}{Corollary}

\theoremstyle{remark}
\newtheorem{Rem}{Remark}
\newtheorem{example}{Examples}

\title[Probabilistic $1$-Lipschitz space ]{The space of probabilistic $1$-Lipschitz maps}
\author{Mohammed Bachir}
\begin{document}

\date{16/01/217} 

\address{Laboratoire SAMM 4543, Universit\'e Paris 1 Panth\'eon-Sorbonne\\
Centre P.M.F. 90 rue Tolbiac\\
75634 Paris cedex 13\\
France}

\email{Mohammed.Bachir@univ-paris1.fr}
%\thanks{Partially supported by grant DGICYT PB98-0618}
\begin{abstract}
We introduce and study the natural notion of probabilistic $1$-Lipschitz maps. We use the space of all probabilistic $1$-Lipschitz maps to give a new method for the construction of probabilistic metric completion (respectively of probabilistic invariant metric group completion). Our construction is of independent interest. We prove that the space of all probabilistic $1$-Lipschitz maps defined on a probabilistic invariant metric group can be endowed with a semigroup structure. Then, we caracterize the probabilistic invariant complete Menger groups by the space of all probabilistic $1$-Lipschitz maps in the sprit of the classical Banach-Stone theorem.

\end{abstract}
\maketitle
\vskip5mm
{\bf msc:} 46S50, 47S50, 54E70.
%\tableofcontents

\section{\bf Introduction}
The general concept of probabilistic metric spaces was introduced by K. Menger, who dealt with probabilistic geometry \cite{M1}, \cite{M2}, \cite{M3}. The decisive influence on the development of the theory of probabilistic metric spaces is due to B. Schweizer and A. Sklar and their coworkers in several papers \cite{SS0}, \cite{SS1}, \cite{SS2}, \cite{SS3}, see also \cite{SHE}, \cite{SHE1} and \cite{HP}. For more informations about this theory we refeer to the excellent monograph \cite{SS}. 
\vskip5mm
The purpose of this paper is to make a new contributions to this theory. We introduce, in a natural way, the concept of probabilistic Lipschitz maps defined from a probabilistic metric space $G$ into the set of all distribution functions that vanish at $0$, classically denoted by $\Delta^+$. We study the properties of the space of all probabilistic $1$-Lipschitz maps denoted by $Lip^1_\star(G,\Delta^+)$ in the paper.  As in the classical determinist analysis, the space of probabilistic $1$-Lipschitz maps will play important role. To motivate the interest of this notion, we mention as example some results obtained about the completions of probabilistic metric spaces. Indeed, the well known result of H. Sherwood in \cite{SHE1}, shows that every probabilistic metric space with continuous triangle function, has a completion. To obtain this result, H. Sherwood used the method of the quotient space by Cauchy sequences. As consequence of our study, we give a new method of constructing of the completions of probabilistic metric spaces by using the concept of the probabilistic $1$-Lipschitz maps instead of the quotient space by Cauchy sequences (see Theorem \ref{thmc} and Corollary \ref{corcomp}). Our method allows at the same time to obtain the completion of a probabilistic invariant metric group (Corollary \ref{corcomp}), thus completing the result of H. Sherwood, just mentioned, in the group framework (in the deterministic case, the group completion for invariant metric groups is a known result of V.L. Klee \cite{K}). We mention here that our method is of independent interest and will be used for proving a probabilistic Banach-Stone type theorem.
\vskip5mm
We will start our study with relatively basic property by giving, for example, a probabilistic analogue of the classical result of E. Mac Shane \cite{MS} on the extension of Lipschitz maps from a subset to the whole space (Theorem \ref{MC}). We prove that the space $Lip^1_\star(G,\Delta^+)$ has also a probabilistic metric structure in which the probabilistic metric space $G$ embeds isometrically . In the case where $G$ is endowed with a group structure, we prove that we can also equip the space $Lip^1_\star(G,\Delta^+)$ with a natural "{\it Sup-Convolution}" operation denoted by $\odot$ so that $(Lip^1_\star(G,\Delta^+),\odot)$ become a probabilistic metric monoid. Then, we prove in a general setting our main results, which implies that the probabilistic invariant complete metric group structure of $G$ is completely determined by the probabilistic metric monoid structure of $(Lip^1_\star(G,\Delta^+),\odot)$. Our results apply in particular to probabilistic invariant complete Menger groups (Corollary \ref{corollaryend}). This gives a probabilistic version of the well known Banach-Stone theorem and extend some results obtained recently by the author for classical invariant metric groups in \cite{Ba1}, \cite{Ba2} and \cite{Ba3}. Recall that the classical Banach-Stone theorem asserts that a compact Hausdorff space $K$ is completely determined by the Banach structure of the space $(C(K), \|.\|_{\infty})$ of all real-valued continuous functions equipped with the sup-norm. In other words, two Hausdorff compact spaces  $K$ and $L$ are homeomorphic if and only if $(C(K), \|.\|_{\infty})$ and $(C(L), \|.\|_{\infty})$ are isometrically isomorphic as Banach spaces. Finally, note that several authors was interested on the extention of classical results obtained in metric spaces to the general framework of probabilistic metric spaces. For example, the theory of fixed points has several extensions in the setting of probabilistic metric spaces (see \cite{HP} and \cite{HP1}). Also, the classical result of Mazur-Ulam on isometries in normed vector spaces has been extended to the probabilistic normed spaces (see for instance \cite{C1}, \cite{C2}).

\vskip5mm

This paper is organized as follows. In Section \ref{S1}, we recall some classical and known notions about probabilistic metric spaces. In Section \ref{S2} we introduce the concept of probabilistic Lipschitz maps and probabilistic sup-convolution, we then deal with some basic properties. Section \ref{S3} and Section \ref{S4} contains the main results of the paper. In Section \ref{S3}, we give our first main results Theorem \ref{thmc}, Theorem \ref{corc} and Corollary \ref{corcomp}. On one hand, these results will permit to give a new construction of the completions of probabilistic metric spaces and probabilistic invariant metric groups. On the other hand, they will play an important role in Section \ref{S4}. In Section \ref{S4}, we give our second main results (Theorem \ref{theorem2}, Theorem \ref{Ba-Sto} and Corollary \ref{corollaryend}), which will permit to state a probabilistic Banach-Stone type theorem.
%We first give an analogueous of the result of Mac Shane (see \cite{MS}) on the extenstion of classical Lipschitz functions from a subset of a metric space to the whole space. In a second time, we study the glogal algebraic structure of the set of all $1$-Lipschitz-probabilistic functions defined on a probabilistic invariant matric group. 

\section{\bf Classical Notions of Probabilistic Metric Spaces.} \label{S1}
In this section, we recall somme general well know definitions and concepts about probabilistic metric spaces. All these concepts can be found in \cite{HP} and \cite{SS}. 
\vskip5mm
A distribution function is a function $F : [-\infty, +\infty] \longrightarrow [0, 1]$ which is nondecreasing
and left-continuous with $F(-\infty) = 0$; $F(+\infty) = 1$. The set of all distribution functions such that
$F(0) = 0$ will be denoted by $\Delta^+$. For $F, G \in \Delta^+$, the relation $F \leq G$ is meant
by $F(t)\leq G(t)$, for all $t\in \R$. For all $a\in \R$, the distribution $\mathcal{H}_a$ is defined as follow
\[ \mathcal{H}_a(t)=
\left\{
\begin{array}{rl}
0& \textnormal{ for all } t\leq a \\
1& \textnormal{ for all } t>a
\end{array}
\right.
\] 
For $a=+\infty$, 
\[ \mathcal{H}_{\infty}(t)=
\left\{
\begin{array}{rl}
0& \textnormal{ for all } t\in [-\infty, +\infty[ \\
1& \textnormal{ for all } t=+\infty
\end{array}
\right.
\] 
It is well known that $(\Delta^+, \leq)$ is a complete lattice with a minimal element $\mathcal{H}_{\infty}$ and the maximal element $\mathcal{H}_0$. Thus, for any nonempty set $I$ and any familly $(F_i)_{i\in I}$ of distributions in $\Delta^+$, the function $F=\sup_{i\in I} F_i$ is also an element of $\Delta^+$. 

\begin{definition} \label{axiom1} (see \cite{SS}, \cite{HP}) A triangle function $\tau$ is a binary operation on $\Delta^+$ that is commutative,
associative, and non-decreasing in each place, and has $\mathcal{H}_0$ as identity. In other words, $\tau$ satisfies the axioms $(i)$-$(v)$: 

$(i)$ $\tau(F, L) \in \Delta^+$ for all $F, L \in \Delta^+$.

$(ii)$ $\tau(F, L)=\tau(L, F)$ for all $F, L \in \Delta^+$.

$(iii)$ $\tau(F,\tau(L, K))=\tau(\tau(F, L), K)$, for all $F,L,K\in \Delta^+$.

$(iv)$ $\tau(F, \mathcal{H}_0)=F$ for all $F\in \Delta^+$.

$(v)$ $F\leq L \Longrightarrow \tau(F, K) \leq \tau(L, K)$ for all $F, L, K\in \Delta^+$.
\end{definition}
\begin{definition} A tiangle function $\tau$ is  said to be sup-continuous (see for instance \cite{C3}) if for all nonempty set $I$ and all familly $(F_i)_{i\in I}$ of distributions in $\Delta^+$ and all $L\in\Delta^+$, we have $$\sup_{i\in I} \tau(F_i, L)=\tau(\sup_{i\in I}(F_i),L).$$
\end{definition}
\vskip5mm
For simplicity of notations, in what follows, the triangle function $\tau$ will be denoted by the binary operation $\star$ as follows:
$$\tau(L,K):=L\star K.$$
It follows from the axioms $(i)$-$(v)$ that $(\Delta^+,\star)$ is an abelian monoid having $\mathcal{H}_0$ as identity element.

\begin{definition} \label{contin} Let $\star$ be a triangle function on $\Delta^+$.

$(1)$ A sequence $(F_n)$ of distributions in $\Delta^+$ converges weakly to a function $F$ in $\Delta^+$ if $(F_n(t))$ converges to $F(t)$ at each point $t$ of continuity of $F$. In this case, we write indifferently $F_n \,{\xrightarrow {\textnormal{w}}}\, F$ or $\lim_n F_n =F$.

$(2)$ We say that the law $\star$ is continuous at $(F,L)\in \Delta^+\times \Delta^+$ if we have $F_n\star L_n \,{\xrightarrow {\textnormal{w}}}\, F\star L$, whenever $F_n \,{\xrightarrow {\textnormal{w}}}\, F$ and $L_n \,{\xrightarrow {\textnormal{w}}}\, L$. 
\end{definition}

\begin{definition} (see \cite{SS}, \cite{HP}) A triangular norm ($t$-norm for short) is a binary operation on the unit interval $[0,1]$, that is, a function $T : [0,1]\times [0,1] \longrightarrow [0,1]$ such that for all $x, y, z \in [0,1]$ the following four axioms are satisfied:

$(T1)$ $T(x,y) = T(y,x)$ ( commutativity);

$(T2)$ $T(x,T(y,z)) = T(T(x,y),z)$ (associativity);

$(T3)$ $T(x,y) \leq T(x,z)$ whenever $y\leq z$ (monotonicity );

$(T4)$ $T(x, 1) = x$ (boundary condition).
\end{definition}
Several basic examples of left-continuous $t$-norm are given in \cite{HP} and \cite{SS}. We give below a classical continuous triangle functions.

\begin{example}\label{example1} (see \cite{HP} and \cite{SS}) Let $T: [0,1]\times [0,1] \longrightarrow [0,1]$ be a left-continuous $t$-norm, then the operation $\star_T$ defined for all $F, L\in \Delta^+$ and for all $t\in \R$ by 
\begin{eqnarray}\label{eq0}
(F\star_T L)(t):=\sup_{s+u=t} T(F(s),L(u))
\end{eqnarray}
satisfies the axioms $(i)$-$(v)$ and is both continuous and sup-continuous. 
\end{example}
%\begin{Rem} In all this paper, we assume that $\star$ is a continuous triangle function satisfying the axioms $(i)$-$(vi)$.
%\end{Rem} 
%In all this work, we assume that $\Delta^+$ is equipped with a continuous law $\star$ satisfying the axiomes $(i)$-$(vi)$.

\begin{definition} \label{definition.M} (see Example \ref{example.1} below) Let $G$ be a set and let $D : G\times G \longrightarrow (\Delta^+,\star, \leq)$ be a map. We say that $(G, D, \star)$ is a probabilistic metric space if the following axioms $(i)$-$(iii)$ hold:

$(i)$ $D(p,q)=\mathcal{H}_0$ iff $p=q$.

$(ii)$ $D(p,q)=D(q,p)$ for all $p,q\in G$ %(Symmetry).

$(iii)$ $D(p,q)\star D(q,r)\leq D(p,r)$ for all $p,q, r\in G$ %(Triangulzr inequality).

If moreover the set $(G,\cdot)$ is a monoid, then we say that $(G,\cdot, D, \star)$ is a probabilistic metric monoide and if $(G,\cdot)$ is group and satisfies the following additional axiom, 

$(iv)$ $D(pr,qr)=D(rp,rq)=D(p,q)$ for all $p,q,r\in G$ %(Invariance).

then, we say that $(G, \cdot, D, \star)$ is an invariant probabilistic metric group.
\end{definition}
\begin{definition} Two probabilistic metric spaces $(G,D,\star)$ and $(G',D',\star)$ (resp. monoids, groups) are said to be isometric (resp. isometrically isomorphic), if there exists a bijective map (resp. isomorphism of monoids, isomorphism of groups) $\mathcal{I}: G\longrightarrow G'$ such that $D'(\mathcal{I}(x),\mathcal{I}(y))=D(x,y)$ for all $x, y\in G$.
\end{definition}
\begin{definition} Let $(G, D,\star)$ be a probabilistic metric space and $\star=\star_T$, where
$$ (K \star_T L)(t) = \sup_{u+v=t}T(K(u), L(v))$$
for a $t$-norm $T$. Then $(G, D,\star)$ is the so called Menger space, which will be denoted
by $(G, D, T)$. If moreover $G$ is a group, we say that $(G, D, T)$ is a Menger group.
\end{definition}

\section{\bf Introduction to the Probabilistic Lipschitz Maps} \label{S2}
We are going to introduce the concept of probabilistic Lipschitz maps defined from a probabilistic metric space into $\Delta^+$. 
\begin{definition} Let $(G, D, \star)$ be a probabilistic metric space and let $f : G \longrightarrow \Delta^+$ be a map. 

$(1)$ We say that $f$ is continuous at $z\in G$ if $f(z_n)\,{\xrightarrow {\textnormal{w}}}\, f(z)$, whenever $D(z_n,z)\,{\xrightarrow {\textnormal{w}}}\,\mathcal{H}_0$. We say that $f$ is continuous if $f$ is continuous at each point $z\in G$.

$(2)$ We say that $f$ is $1$-Lipschitz map if: $$\forall x, y \in G,\hspace{1mm} D(x,y)\star f(y)\leq f(x).$$
\end{definition}
%\begin{Rem}
\vskip5mm
We can also define a $k$-Lipschitz map $f$ for a nonegative real number $k\geq 0$ as follows 
 $$\forall x, y \in G,\hspace{1mm} D_k(x,y)\star f(y)\leq f(x),$$
where, for all $x, y \in G$ and all $t\in \R$, $D_k(x,y)(t)=D(x,y)(\frac{t}{k})$ if $k>0$ and $D_0(x,y)(t)=\mathcal{H}_0(t)$ if $k=0$. In this paper, we deal only with $1$-Lipschitz maps.
%\end{Rem}
%Note that $f$ is $k$-Lipschitz if and only $f_k$ is $1$-Lipschitz where for all $x\in G$ and all $t\in \R$: $f_k(x)(t):=f(\frac{t}{k})$.
\begin{example} \label{example.2} Let $(G,d)$ be a metric space. Assume that $\star$ is a triangle function on $\Delta^+$ satisfying $\mathcal{H}_a\star \mathcal{H}_b=\mathcal{H}_{a+b}$ for all $a, b\in \R^+$ (for example if $\star=\star_T$ where $T$ is a lef-continuous $t$-norm). Let $(G, D, \star)$ be the probabilistic metric space defined with the probabilistic metric 
$$D(p,q)=\mathcal{H}_{d(p,q)}.$$
Let $L: (G,d) \longrightarrow \R^+$ be a real-valued map. Then, $L$ is a nonnegative $1$-Lipschitz map if and only if $f : (G,D,\star) \longrightarrow \Delta^+$ defined for all $x\in G$ by
$$f(x):=\mathcal{H}_{L(x)}$$
is a probabilistic $1$-Lipschitz map. This example shows that the framework of probabilistic $1$-Lipschitz maps encompasses the classical determinist case.

\end{example}
\vskip5mm
By $Lip^1_\star(G,\Delta^+)$ we denote the space of all probabilistic $1$-Lipschitz maps   

$$Lip^1_\star(G,\Delta^+):=\lbrace f : G\longrightarrow \Delta^+/ D(x,y)\star f(y)\leq f(x); \forall x, y \in G\rbrace.$$

For all $x\in G$, by $\delta_x$ we denote the map 
\begin{eqnarray*}
\delta_x : G &\longrightarrow& \Delta^+\\
            y&\mapsto& D(y,x).
\end{eqnarray*}

We set $\mathcal{G}(G):=\lbrace \delta_x/ x\in G \rbrace$ and by $\delta$, we denote the operator  
\begin{eqnarray*}
\delta : G &\longrightarrow& \mathcal{G}(G)\subset Lip^1_\star(G,\Delta^+)\\
            x &\mapsto& \delta_x.
\end{eqnarray*}
\subsection{Basic Properties of $1$-Lipschitz Maps} 
In this section, we give some basic and elementary but useful properties of probabilistic $1$-Lipschitz maps.
%
%
%Let $f, g : G \longrightarrow \Delta^+$ be two maps. The distribution function $\D(f,g)\in \Delta^+$ is defined from $\R$ to $[0,1]$ as follows: 
%$$ \D (f,g):=\sup_{x\in G} (f(x)\star g(x)).$$ 
\begin{proposition} \label{prop9} Let $(G, D, \star)$ be a probabilistic metric space. Then, we have that $\delta_a \in Lip^1_\star(G,\Delta^+)$ for each $a\in G$ and the map $\delta : G \longrightarrow Lip^1_\star(G,\Delta^+)$ is injective. 
\end{proposition}
\begin{proof} The fact that $\delta_a \in Lip^1_\star(G,\Delta^+)$ for each $a\in G$ follows from the property: $ D(x,y)\star D(y,a) \leq D(x,a)$ for all $a, x, y\in G$. Now, let $a, b \in G$ be such that $\delta_a=\delta_b$. It follows that $\delta_a(x)=\delta_b(x)$ for all $x\in G$. In particular, for $x=b$ we have that $D(a,b)=\delta_a(b)=\delta_b(b)=\mathcal{H}_0$, which implies that $a=b$.\\
\end{proof}
\vskip5mm
Let $f\in Lip^1_\star(G,\Delta^+)$ and $F\in \Delta^+$, by $\langle f, F\rangle : G \longrightarrow \Delta^+$, we denote the map defined by $\langle f, F\rangle (x):=f(x)\star F$ for all $x\in G$. We easly obtain the following propostion.

\begin{proposition} Let $(G, D, \star)$ be a probabilistic metric space. Then, for all $f \in Lip^1_\star(G,\Delta^+)$ and all $F\in \Delta^+$, we have that $\langle f, F\rangle \in Lip^1_\star(G,\Delta^+)$.
\end{proposition}
In a classical metric space $(M,d)$, we know that the distance function to a nonempty subset $A$ of $M$ defined by $x\mapsto d(x,A):=\inf_{y\in A} d(x,y)$ is a real-valued $1$-Lipschitz map. We have the analogous of this property in the probabilistic metric case.
\begin{proposition} Let $(G,D,\star)$ be a probabilistic metric space. Suppose that $\star$ is sup-continuous. Let $A$ be a nonempty subset of $G$. Then the map $D(x,A): x\mapsto \sup_{y\in A} D(x,y)$ is a probabilistic $1$-Lipschitz map.
\end{proposition}
\begin{proof} Using the sup-continuity of $\star$, we have for all $a, b \in G$
\begin{eqnarray*}
D(a,b)\star D(a,A):=D(a,b)\star \sup_{y\in A} D(a,y)&=&\sup_{y\in A} (D(a,b)\star D(a,y)\\
                                                    &\leq& \sup_{y\in A} D(b,y) =D(b,A).
\end{eqnarray*}
\end{proof}
The following classical proposition will be used very often:
\begin{proposition} \label{util}
Let $(F_n), (L_n), (K_n) \subset (\Delta^+,\star)$. Suppose that 

$(a)$ the triangle function $\star$ is continuous,

$(b)$ $F_n\,{\xrightarrow {\textnormal{w}}}\, F$, $L_n\,{\xrightarrow {\textnormal{w}}}\, L$ et $K_n\,{\xrightarrow {\textnormal{w}}}\, K$. 

$(c)$ for all $n\in \N$, $F_n \star L_n \leq K_n$,

then, $F\star L\leq K$.
\end{proposition}
\begin{proof} It is well known that an increasing function from $\R$ into $\R$ is continuous outside a set which is at most countable. Clearly, from Definition \ref{contin} we have that $(F\star L)(t)\leq K(t)$ for all $t\in \R$ which is a point of continuity of both $F\star L$ and $K$. Since both $F\star L$ and $K$ are let-continuous, we deduce that $(F\star L)(t)\leq K(t)$ for all $t\in \R$.
\end{proof}
We also recall the following deep results due to D. Sibley: From \cite[Theorem 1. and Theorem 2.]{DS} we know that the topological space $(\Delta^+,\textnormal{w})$ (equipped with the weak convergence) is metrizable by the modified L\'evy metric and from \cite[Theorem 3.]{DS} we know that $(\Delta^+,\textnormal{w})$ is a compact metrizable space. This last result will be used very often in this paper. We have the following proposition.
\begin{proposition} Let $(G,D,\star)$ be a probabilistic metric space such that $\star$ is continuous. Then, every probabilistic $1$-Lipschitz map defined on $G$ is continuous.
\end{proposition}

\begin{proof} Let $f$ be a probabilistic $1$-Lipschitz map, $x\in G$ and $(x_n)_n\subset G$ be a sequence such that $\lim_n D(x_n,x)=\mathcal{H}_0$. Since $(\Delta^+,\textnormal{w})$ is a compact metrizable space, it suffices to prove that every (weak) convergent subsequence $(f(x_{\varphi(n)}))$, converges necessary to $f(x)$. Suppose that a subsequence $(f(x_{\varphi(n)}))$ converges to some distribution $F$ and let us prove that $F=f(x)$. Indeed, Since $f$ is a probabilistic $1$-Lipschitz map, then for all $n \in \N$ we have
$$D(x_{\varphi(n)},x)\star f(x)\leq f(x_{\varphi(n)}).$$
Using Proposition \ref{util} and the fact that $D(x_{\varphi(n)},x)\,{\xrightarrow {\textnormal{w}}}\,\mathcal{H}_0$, we get from the above inequality that $f(x)\leq F$. Also we have 
$$D(x_{\varphi(n)},x)\star f(x_{\varphi(n)})\leq f(x).$$
Again, using Proposition \ref{util}, the fact that $D(x_{\varphi(n)},x)\,{\xrightarrow {\textnormal{w}}}\,\mathcal{H}_0$ and $f(x_{\varphi(n)})\,{\xrightarrow {\textnormal{w}}}\, F$, we get that $F\leq f(x)$. Thus $F=f(x)$ and so, $f(x_n)\,{\xrightarrow {\textnormal{w}}}\,f(x)$. This finish the proof.
\end{proof}

\subsection{A Probabilistic Analogue of Mac Shane Result's} 

Recall the Lipschitz extention result of Mac Shane in \cite{MS}: if $(X,d)$ is a metric space, $A$ a nonempty subset of $X$ and $f: A \longrightarrow \R$ is $k$-Lipschitz map, then there exist a $k$-Lipschitz map $\tilde{f}: X \longrightarrow \R$ such that $\tilde{f}_{|A}=f$. We give bellow an analogue of this result for probabilistic $1$-Lipschitz maps.

\begin{theorem} \label{MC} Let $(G, D, \star)$ be a probabilistic metric space such that $\star$ is sup-continuous and let $A$ be a nonempty subset of $G$. Let $f: A\longrightarrow \Delta^+$ be a probabilistic $1$-Lipschitz map. Then, there exists a probabilistic $1$-Lipschitz map $\tilde{f}: G\longrightarrow \Delta^+$ such that $\tilde{f}_{|A}=f$.
\end{theorem}
\begin{proof}
We define $\tilde{f}: G \longrightarrow \Delta^+$ as follows: for all $x\in G$,
$$\tilde{f}(x):=\sup_{a\in A} D(a,x)\star f(a).$$
We first prove that $\tilde{f}(x)=f(x)$ for all $x\in A$. Indeed, let $x\in A$. On one hand we have $f(x)=\mathcal{H}_0\star f(x)=D(x,x)\star f(x)\leq \sup_{a\in A} D(a,x)\star f(a)=\tilde{f}(x)$. On the other hand, since $f$ is probabilistic $1$-Lipschitz on $A$ and $x\in A$, then $D(a,x)\star f(a)\leq f(x)$ for all $a\in A$. It follows that $\tilde{f}(x):=\sup_{a\in A} D(a,x)\star f(a) \leq f(x)$. Thus, $\tilde{f}(x)=f(x)$ for all $x\in A$. Now, we show that $\tilde{f}$ is probabilistic $1$-Lipschitz on $G$. Indeed, Let $x, y \in G$. For all $a,x,y\in A$ we have that $D(a,x)\star D(x,y)\leq D(a,y)$. So, $D(a,x)\star f(a)\star D(x,y)\leq D(a,y)\star f(a)$. By taking the supremum over $a\in A$ and using the sup-continuity of $\star$ we get $\tilde{f}(x)\star D(x,y) \leq \tilde{f}(y)$, for all $x,y \in G$. Hence, $\tilde{f}$ is probabilistic $1$-Lipschitz map on $G$ that coincides with $f$ on $A$.
\end{proof}
%\subsection{The probabilistic metric structure of the space $Lip^1_\star(G,\Delta^+)$}
%
\subsection{Probabilistic Sup-Convolution} 
In this section, we assume that $G$ is a group. The aim of this section is to prove that if $(G, \cdot, D, \star)$ is a  probabilistic invariant metric group having $e$ as identity element, then we can equip the set $Lip^1_\star(G,\Delta^+)$ of probabilistic $1$-Lipschitz maps with a Sup-Convolution law denoted by $\odot$, so that $(Lip^1_\star(G,\Delta^+), \odot)$ becomes a monoid having $\delta_e$ as identity element (Theorem \ref{thm1}). This will allow, subsequently, to embed $G$ as group into $Lip^1_\star(G,\Delta^+)$. The probabilistic sup-convolution introduced below is an extension of the inf-convolution treated recently in the classical invariant metric groups framework in \cite{Ba1}, \cite{Ba2} and \cite{Ba3}.
\vskip5mm
%\begin{definition} 
Let $G$ be a group and let $f, g : G \longrightarrow \Delta^+$ be two maps. The Sup-Convolution of $f$ and $g$ denoted $f\odot g :G \longrightarrow \Delta^+$ is a map defined for all $x\in G$ as follows 
\begin{eqnarray*}
(f \odot g) (x)&:=& \sup_{y,z\in G / yz=x} f(y)\star g(z)\\
                  &=& \sup_{y\in G} f(y)\star g(y^{-1}x).
\end{eqnarray*} 
%\end{definition}
%
%
In order to prove that $(Lip^1_\star(G,\Delta^+), \odot)$ has a monoid structure, we need to establish the three following propositions. Recall that a monoid is a semigroup having an identity element.
\begin{proposition} \label{prop10} Let $(G, \cdot, D, \star)$ be an invariant probabilistic metric group. Then, we have that $\delta_a\odot \delta_b=\delta_{ab}$ for all $a,b\in G$. 
\end{proposition}
\begin{proof} We first prove that $\delta_e\odot \delta_e=\delta_e$, where $e$ is the identity element of $G$. Indeed, for all $x\in G$,
\begin{eqnarray*}
(\delta_e \odot \delta_e) (x)&:=& \sup_{y,z\in G / yz=x} D(y,e)\star D(z,e)\\
                                &=& \sup_{y,z\in G / yz=x}  D(y,e)\star D(e,z^{-1}) \textnormal{ by the invariance of } D\\
                                &\leq& \sup_{y,z\in G / yz=x} D(y,z^{-1}) \textnormal{ by the property $(iii)$ of } D\\
                                &=& \sup_{y,z\in G / yz=x} D(yz,e) \textnormal{ by the invariance of } D\\
                                &=& D(x,e).\\
                                &=& \delta_e(x)
\end{eqnarray*}
On the other hand, 
\begin{eqnarray*}
(\delta_e \odot \delta_e) (x)&:=& \sup_{y,z\in G / yz=x} D(y,e)\star D(z,e)\\
                                &\geq& D(x,e)\star D(e,e)\\
                                &=& D(x,e)\star \mathcal{H}_0\\
                                &=& D(x,e)\\
                                &=& \delta_e(x)
\end{eqnarray*}
Now, we prove that that $\delta_a\odot \delta_b=\delta_{ab}$ for all $a,b\in G$. Indeed, let $x\in G$ and $t\in \R$, using the change variable $u=zb^{-1}$ and the invariance of $D$ we have:
\begin{eqnarray*} 
\delta_{a}\odot \delta_{b}(x) & = & \sup_{z\in X} \delta_{a}(xz^{-1})\star\delta_{b}(z) \\
                                                       &=& \sup_{z\in G} \delta_e(xz^{-1}a^{-1})\star\delta_e(zb^{-1}) \\
                                                       &=& \sup_{u\in G} \delta_e(xb^{-1}t^{-1}a^{-1})\star\delta_e(u)\\
                                                     &=& \sup_{u\in X} \delta_e(a^{-1}xb^{-1}u^{-1})\star\delta_e(u)\\
                                                     &=& (\delta_e\odot \delta_e)(a^{-1}xb^{-1})\\
                                                       &=& \delta_e(a^{-1}xb^{-1})\\
                                                       &=& \delta_e(xb^{-1}a^{-1})\\
                                                       &=& \delta_e(x(ab)^{-1})\\
                                                       &=& \delta_{ab}(x).
\end{eqnarray*}
Thus, $\delta_a\odot \delta_b=\delta_{ab}$ for all $a,b\in G$.
\end{proof}
When $G$ is a probabilistic invariant metric group, the space $Lip^1_\star(G,\Delta^+)$ turns out to be naturally linked to the operation of Sup-Convolution. It is characterized by being the maximal set of maps from $G$ into $\Delta^+$ having $\delta_e$ as identity element for the operation $\odot$.
\begin{proposition} \label{prop12} Let $(G, \cdot, D, \star)$ be an invariant probabilistic metric group. Then, we have $$Lip^1_\star(G,\Delta^+)= \lbrace f : G\longrightarrow \Delta^+ \textnormal{ map } / f\odot \delta_e =\delta_e\odot f=f\rbrace.$$
%$(Lip^1_\star(G,\Delta^+),\odot)$ is a monoid having $\delta_e$ as identity element (where $e$ is the identity element of $G$). 
%and $(Lip^1_\star(G,\Delta^+),\D, \star)$ is a probabilistic metric monoid. 
\end{proposition}
\begin{proof} 
Let $f\in Lip^1_\star(G,\Delta^+)$. We first prove that $\delta_e\odot f(x)=f\odot \delta_e(x)$ for all $x\in G$. Using the variable change $u=xy^{-1}$ and the invariance of $D$, we have for all $x\in G$,
 \begin{eqnarray} (f\odot \delta_e)(x)&=&\sup_{y\in G} f(xy^{-1})\star\delta_e(y) \nonumber\\
                                    &=& \sup_{u\in G} f(u)\star\delta_e(u^{-1}x) \nonumber\\
                                    &=& \sup_{u\in G} \delta_e(xu^{-1})\star f(u) \nonumber\\
                                    &=& (\delta_e\odot f)(x). \nonumber
 \end{eqnarray} 
Now, let $f\in Lip^1_\star(G,\Delta^+)$ then for all $x, y \in G$ we have
\begin{eqnarray*}
f(x)&\geq& d(x,y)\star f(y)\\
    &=& d(xy^{-1},e)\star f(y)\\
    &=& \delta_e(xy^{-1})\star f(y).
\end{eqnarray*}
Taking the supremum over $y\in G$, we get $f(x)\geq (\delta_e\odot f)(x)=(f\odot \delta_e)(x)$ for all $x\in G$. For the converse inequality we have 
 \begin{eqnarray} f\odot \delta_e(x) &=& \sup_{y\in G} f(xy^{-1})\star \delta_e(y) \nonumber\\
                                                     &\geq & f(x)\star \delta_e(e) \nonumber\\
                                                     &=& f(x)\star \mathcal{H}_0 \nonumber\\
                                                     &=& f(x).\nonumber
 \end{eqnarray}
Thus $f\odot \delta_e =\delta_e\odot f=f$ for all $f\in Lip^1_\star(G,\Delta^+)$ and so we have that $Lip^1_\star(G,\Delta^+)\subset \lbrace f : G\longrightarrow \Delta^+ \textnormal{ map } / f\odot \delta_e =\delta_e\odot f=f\rbrace$. Now, let $f\in \lbrace f : G\longrightarrow \Delta^+ \textnormal{ map } / f\odot \delta_e =\delta_e\odot f=f\rbrace$. It follows that for all $x\in G$ 
 \begin{eqnarray*}
f(x) &=& f\odot \delta_e (x)\\
        &=& \sup_{yz=x} f(y)\star \delta_e(z)\\
        &\geq& f(y)\star \delta_e(y^{-1}x)\\
        &=& D(x,y)\star f(y).
\end{eqnarray*} 
It follows that $f\in Lip^1_\star(G,\Delta^+)$. Finally we have $Lip^1_\star(G,\Delta^+)= \lbrace f : G\longrightarrow \Delta^+ \textnormal{ map } / f\odot \delta_e =\delta_e\odot f=f\rbrace$.
\end{proof}
%\vskip5mm
%Now, we assume that there exists a binary operation operation $\overline{\star}$ defined on $\R\times \R$ associated to the operation $\star$ on $\Delta^+$, in the following sens: for all $F, G\in \Delta^+$ and all $t\in \R$ we have $(F\star G) (t)=F(t)\overline{\star} G(t)$. 
\begin{proposition} \label{prop13} Let $(G, \cdot, D, \star)$ be an invariant probabilistic metric group. Suppose that $\star$ is sup-continuous. Then, the set $(Lip^1_\star(G,\Delta^+), \odot)$ has the associative property. 
%$(Lip^1_\star(G,\Delta^+),\odot)$ is a monoid having $\delta_e$ as identity element (where $e$ is the identity element of $G$). 
%and $(Lip^1_\star(G,\Delta^+),\D, \star)$ is a probabilistic metric monoid. 
\end{proposition}
\begin{proof} Let $f, g, h \in Lip^1_\star(G,\Delta^+)$. Using the variable change $u:=yz$ together with the sup-continuity of $\star$, we obtain for all $x\in G$,
\begin{eqnarray*}
(f\odot g)\odot h(x)&=&\sup_{z\in G}\lbrace (f\odot g)(xz^{-1})\star h(z)\rbrace\\
                      &=&\sup_{z\in G}\lbrace (f\odot g)(xz^{-1})\star h(z)\rbrace\\
                      &=& \sup_{z\in G}\lbrace \sup_{y\in G}\lbrace f(xz^{-1}y^{-1})\star g(y)\rbrace\star h(z)\rbrace\\
                      &=& \sup_{z\in G}\lbrace \sup_{y\in G}\lbrace f(xz^{-1}y^{-1})\star g(y)\rbrace\star h(z)\rbrace\\
                      &=& \sup_{z\in G}\lbrace \sup_{u\in G}\lbrace f(xu^{-1})\star g(uz^{-1})\rbrace \star h(z)\rbrace\\
                      &=& \sup_{u\in G}\lbrace f(xu^{-1})\star \sup_{z\in G} \lbrace g(tz^{-1})\star h(z)\rbrace\rbrace\\
                      &=& \sup_{u\in G}\lbrace f(xu^{-1})\star (g\odot h)(u)\rbrace\\
                      &=& f\odot(g\odot h)(x).
\end{eqnarray*}
\end{proof}
\begin{theorem} \label{thm1} Let $(G, \cdot, D, \star)$ be an invariant probabilistic metric group. Suppose that $\star$ is sup-continuous. Then, the set $(Lip^1_\star(G,\Delta^+),\odot)$ is a monoid having $\delta_e$ as identity element. Moreover, the map $\delta: (G, \cdot) \longrightarrow (Lip^1_\star(G,\Delta^+),\odot)$ is an injective group homomorphism.
%$(Lip^1_\star(G,\Delta^+),\odot)$ is a monoid having $\delta_e$ as identity element (where $e$ is the identity element of $G$). 
%and $(Lip^1_\star(G,\Delta^+),\D, \star)$ is a probabilistic metric monoid. 
\end{theorem}
\begin{proof} First we prove that $f\odot g \in Lip^1_\star(G,\Delta^+)$, whenever $f, g \in Lip^1_\star(G,\Delta^+)$. Using Proposition \ref{prop12},it suffices to shows that $\delta_e\odot (f\odot g)=(f\odot g)\odot \delta_e=f\odot g$. Indeed, since $f\in Lip^1_\star(G,\Delta^+)$, then by Proposition \ref{prop12} we have that $\delta_e \odot f=f$. Thus, $(\delta_e \odot f)\odot g=f\odot g$. Using the associative property (Proposition \ref{prop13}) we obatin that $\delta_e \odot (f\odot g)=f\odot g$. In a similar way, since $g\odot \delta_e=g$, we obtain (using the associative property) that  $(f\odot g)\odot \delta_e=f\odot g$. Finally, we have $\delta_e\odot (f\odot g)=(f\odot g)\odot \delta_e=f\odot g$, which implies that $f\odot g \in Lip^1_\star(G,\Delta^+)$. This property with the associative property shows that $(Lip^1_\star(G,\Delta^+), \odot)$ is a semigroup. Using again Proposition \ref{prop12}, we obtain that $(Lip^1_\star(G,\Delta^+), \odot)$ is a monoid having $\delta_e$ as identity element. Using Proposition \ref{prop10} we get that the map $\delta: (G, \cdot) \longrightarrow (Lip^1_\star(G,\Delta^+),\odot)$ is a group homeomorphism.
\end{proof}
\section{\bf The Probabilistic Metric $1$-Lipschitz Space and Completions} \label{S3}
In this section, we give a new method of constructing the completion of probabilistic metric space. This method is based on the use of the probabilistic $1$-Lipschitz maps and is of independ interest. Indeed, it will be necessary to treat Section \ref{S4} especially to prove theorem \ref{theorem2}. Recall that in the classical metric spaces, there are at least two methods for constructing the completion of a metric space $M$. The first one uses the quotient by the Cauchy sequences of $M$ and the second one uses the isometric embedding of $M$ into a complete metric space of functions. The proof, in the probabilistic case, given by H. Sherwood in \cite{SHE1} uses the first approach. We will follow here the second approach, by embedding isometrically a given probabilistic metric space $G$ into a probabilistic metric space of $1$-Lipschitz maps. Note also that the completion of probabilistic normed space was given by B. Lafuerza Guill´en, J. A. Rodriguez Lallena, and C. Sempi \cite{LL}, using also the method of the quotient space by Cauchy sequences. In this section, we give in particular (Corollary \ref{corcomp}) the new result on the group completion of any probabilistic invariant metric group, which is the probabilistic version of the group completion given by V.L. Klee \cite{K} in the classical metric case.

%This section is divided on two subsection. The big part concerns the general case of probabilistic metric space. A second part concerne the prbabilistic invariant metric groups.
%\subsection{The probabilistic complete metric space $(\Pi(G), \D,\star)$.}
\vskip5mm
We recall below the definition of probabilistic complete metric space.

\begin{definition} Let $(G, D, \star)$ be a probabilistic metric space. A sequence $(z_n)\subset G$ is said to be a Cauchy sequence if for all $t\in \R$,
$$\lim_{n,p \longrightarrow +\infty} D(z_n, z_p)(t)=\mathcal{H}_0(t).$$
(Equivalently, if $D(z_n, z_p)\,{\xrightarrow {\textnormal{w}}}\,\mathcal{H}_0$, when $n, p\longrightarrow +\infty$). A probabilistic metric space $(G, D, \star)$ is said to be complete if every Cauchy sequence $(z_n)\subset G$ weakly converges to some $z_{\infty}\in G$, that is $\lim_{n \rightarrow +\infty} D(z_n, z_{\infty})(t)=\mathcal{H}_0(t)$ for all $t\in \R$, we will briefly note $\lim_n D(z_n, z_{\infty})=\mathcal{H}_0$. 
\end{definition}
\begin{example} \label{example.1} We give below general frameworks of probabilistic complete metric spaces and probabilistic invariant complete metric groups.

$(1)$ Every complete metric space induce a probabilistic complete metric space. Indeed, if $d$ is a metric on $G$ and $\star$ is a triangle function on $\Delta^+$ satisfying $\mathcal{H}_a\star \mathcal{H}_b=\mathcal{H}_{a+b}$ for all $a, b\in \R^+$ (the triangle function $\star_T$ with a left-continuous $t$-norm $T$, satisfies this condition, see for instance \cite{HP} and \cite{SS}), then $(G, D, \star)$ where,
$$D(p,q)=\mathcal{H}_{d(p,q)}, \hspace{1mm} \forall p,q \in G, $$ 
is a probabilistic complete metric space. Thus, if $(G,\cdot,d)$ is invariant complete metric group (Several examples of invariant complete metric groups are given in \cite{Ba2}) then $(G,\cdot,D,\star)$ is a probabilistic invariant complete metric group.

$(2)$ Let $(G,F,\star)$ be a probabilistic complete normed vector space with a probabilistic norm $F$ (the definition of probabilistic complete normed space can be found in \cite{AS}, see also \cite{HP} p.66 for example). For all $p, q\in G$, let us set 
$$D(p,q)=F_{p-q}.$$
Then, $(G,+,D,\star)$ is a probabilistic invariant complete metric group (this example also works by remplacing probabilistic normed vector space by probabilistic normed groups, the definition can be found in \cite{SP},\cite{NP}).
\end{example}

\begin{definition}
Let $(G,D,\star)$ be a probabilistic metric space. We define the subset $\Pi(G)$ of $Lip^1_\star(G,\Delta^+)$ as follows: $f\in \Pi(G)$ if and only if, $f$ is a probabilistic $1$-Lipschitz map such that, there exist a Cauchy sequence $(a_n)\subset G$ satisfying: for all $x\in G$, $\delta_{a_n}(x):= D(a_n,x)\,{\xrightarrow {\textnormal{w}}}\, f(x)$, when $n\rightarrow+\infty$.
\end{definition}
Note that $\mathcal{G}(G)$ is always a subset of $\Pi(G)$ and we have $\mathcal{G}(G)=\Pi(G)$ if and only if $G$ is complete. The main result of this section is the following theorem. Its proof will be given after two lemmas.
\begin{theorem} \label{thmc} Let $(G,D,\star)$ be any probabilistic metric space, where $\star$ is a continuous triangle function. Let $\D: \Pi(G)\times \Pi(G)\longrightarrow \Delta^+$ be the map defined by
\begin{eqnarray*}
\D (f,g)=\sup_{x\in G} f(x)\star g(x)
\end{eqnarray*}
Then,

$(1)$ for all $a, b \in G$, $\D(\delta_a, \delta_b)=D(a,b)$.

$(2)$ $(\Pi(G), \D,\star)$ is a probabilistic complete metric space.

$(3)$ $\Pi(G)= \overline{\mathcal{G}(G)}^{\D}$. In other words, $\delta(G):=\mathcal{G}(G)$ is dense in $(\Pi(G), \D,\star)$.

\end{theorem}

\begin{lemma} \label{Lelemme} Let $(G,D,\star)$ be a probabilistic metric space (not assumed to be complete), where $\star$ is a continuous triangle function. Let $(a_n)\subset G$ be a Cauchy sequence. Then, there exists a probabilistic $1$-Lipschitz map $f: G\longrightarrow \Delta^+$ such that, for all $x\in G$ we have $D(a_n,x)\,{\xrightarrow {\textnormal{w}}}\, f(x)$, when $n\longrightarrow +\infty$. In particular, $f(a_n)\,{\xrightarrow {\textnormal{w}}}\, \mathcal{H}_{0}$, when $n\longrightarrow +\infty$.
\end{lemma}
\begin{proof} We know from \cite{DS} that $(\Delta^+,\textnormal{w})$ is a compact metrizable space. Thus, for each fixed point $x\in G$, the sequence $(D(a_n,x))$ has a convergent subsequence that is, there exists $f(x) \in \Delta^+$ and a strictly increasing map $\varphi_x : \N\longrightarrow \N$ such that $D(a_{\varphi_x(n)},x)\,{\xrightarrow {\textnormal{w}}}\, f(x)$ when $n\longrightarrow +\infty$. From the properties of $D$, we have for all $y\in G$,
\begin{eqnarray} \label{comp1}
D(a_{\varphi_y(n)},y)\star D(a_{\varphi_x(n)},a_{\varphi_y(n)}) \leq D(a_{\varphi_x(n)},y)
\end{eqnarray}
Also, for each $y\in G$ the sequence $(D(a_{\varphi_x(n)},y))$ has a convergent subsequence by the compactness of $(\Delta^+,\textnormal{w})$ that is, there exists $g_x(y)\in \Delta^+$ and there exists a strictly increasing map $\psi_{x,y} : \N\longrightarrow \N$ such that
\begin{eqnarray}\label{comp2}
D(a_{\varphi_x(\psi_{x,y}(n))},y)\,{\xrightarrow {\textnormal{w}}}\, g_x(y).
\end{eqnarray}
when $n\longrightarrow +\infty$. Using (\ref{comp1}), (\ref{comp2}), the continuity of $\star$, the fact that $(a_n)$ is Cauchy sequence and taking the limit (see Proposition \ref{util}), we get
\begin{eqnarray}\label{comp3}
f(y)=f(y)\star \mathcal{H}_0 \leq g_x(y).
\end{eqnarray}
On the other hand, we have 
\begin{eqnarray*}
D(y,x)\star D(a_{\varphi_x(\psi_{x,y}(n))}, y)&\leq&  D(a_{\varphi_x(\psi_{x,y}(n))}, x)\nonumber\\
\end{eqnarray*}
So by Proposition \ref{util},
\begin{eqnarray}\label{comp4'}
D(y,x)\star g_x(y)&\leq&  f(x)
\end{eqnarray}
Thus, from (\ref{comp3}) and (\ref{comp4'}) we have for all $x, y \in G$ 
\begin{eqnarray*}\label{comp4}
D(y,x)\star f(y) &\leq& f(x).
\end{eqnarray*}
This shows that the well defined map $f$ is a probabilistic $1$-Lipschitz map. Now, we are going to prove that $D(a_n,x)\,{\xrightarrow {\textnormal{w}}}\, f(x)$ for all $x\in G$, when $n \longrightarrow +\infty$. Indeed, let us fixe an $x\in G$. In order to prove that $D(a_n,x)\,{\xrightarrow {\textnormal{w}}}\, f(x)$, we need only to prove that the sequence $(D(a_n,x))$ has a unique cluster point, since $(\Delta^+,\textnormal{w})$ is a metrizable compact set. Let $(D(a_{\varphi(n)},x))$ be a convergent subsequence to some $\overline{f}(x)$. Let us prove that necessarily $\overline{f}(x)=f(x)$. Indeed, from the properties of $D$, we have that
\begin{eqnarray*}\label{comp5}
D(a_{\varphi_x(n)},x)\star D(a_{\varphi_x(n)},a_{\varphi(n)}) \leq D(a_{\varphi(n)},x).
\end{eqnarray*}
Since $D(a_{\varphi_x(n)},x)\,{\xrightarrow {\textnormal{w}}}\, f(x)$ and $(a_n)$ is a Cauchy sequence, then from Proposition \ref{util}
\begin{eqnarray*}\label{comp6}
f(x)=f(x)\star \mathcal{H}_0 \leq \overline{f}(x).
\end{eqnarray*}
In a similar way, we have 
\begin{eqnarray*}
D(a_{\varphi(n)},x)\star D(a_{\varphi(n)},a_{\varphi_x(n)}) \leq D(a_{\varphi_x(n)},x).
\end{eqnarray*}
Since $D(a_{\varphi(n)},x)\,{\xrightarrow {\textnormal{w}}}\, \overline{f}(x)$ and $(a_n)$ is a Cauchy sequence, then from Proposition \ref{util}
\begin{eqnarray*}
\overline{f}(x)=\overline{f}(x)\star \mathcal{H}_0\leq f(x).
\end{eqnarray*}
Thus $f(x)=\overline{f}(x)$ and so, $D(a_n,x)\,{\xrightarrow {\textnormal{w}}}\, f(x)$. Now, to see that $f(a_n)\,{\xrightarrow {\textnormal{w}}}\, \mathcal{H}_{0}$ (when $n\longrightarrow +\infty$), it suffices to observe that $\lim_n f(a_n)=\lim_n (\lim_p D(a_p,a_n))=\mathcal{H}_0$, since $(a_n)$ is a Cauchy sequence.
\end{proof}
\begin{lemma} \label{Lelemme2} Let $(G,D,\star)$ be a probabilistic metric space such that $\star$ is continuous. Let $f, g\in \Pi(G)$ and let $(a_n), (b_n)\subset G$ be any Cauchy sequences such that $D(a_n, x)\,{\xrightarrow {\textnormal{w}}}\, f(x)$, and $D(b_n,x) \,{\xrightarrow {\textnormal{w}}}\,g (x)$ for all $x\in G$, when $n\longrightarrow +\infty$. Then, $\lim_{n, m} D(a_n, b_m)$ exists and is independent from the choice of the Cauchy sequences $(a_n)$ and $(b_n)$. More precisely, we have that $$\sup_{x\in G} f(x)\star g(x)=\lim_{n, m} D(a_n, b_m)=\lim_{n} D(a_n, b_n)=\lim_m f(b_m)=\lim_n g(a_n).$$
\end{lemma}
\begin{proof} Let $(a_n), (b_n)\subset G$ be any Cauchy sequences such that $D(a_n, x)\,{\xrightarrow {\textnormal{w}}}\, f(x)$, and $D(b_n,x) \,{\xrightarrow {\textnormal{w}}}\,g (x)$ for all $x\in G$, when $n\longrightarrow +\infty$. In particular we have that $f(a_n)\,{\xrightarrow {\textnormal{w}}}\, \mathcal{H}_0$ and $g(b_n)\,{\xrightarrow {\textnormal{w}}}\, \mathcal{H}_0$, when $n \longrightarrow +\infty$ (see Lemma \ref{Lelemme}). Since $f$ is a probabilistic $1$-Lipschitz then, for all $m\in \N$
\begin{eqnarray*}
f(x)\star D(b_m,x)\leq& f(b_m).
\end{eqnarray*}
We prove that $\lim_m f(b_m)=\sup_{x\in G}f(x)\star g(x)$. To do this, it suffices to show that every convergent subsequence of $(f(b_m))$ converges to $\sup_{x\in G}f(x)\star g(x)$ since $(\Delta^+,\textnormal{w})$ is a compact metrizable space. Indeed, let $(f_{\varphi(m)})$ be a convergent subsequence to some $F\in \Delta^+$. By using the continuity of $\star$ and taking the limit over $m\in \N$ in the above inequality (using Proposition \ref{util}), we get that for all $x\in G$, $f(x)\star g(x) \leq F$.  So,
$$\sup_{x\in G}f(x)\star g(x) \leq F.$$
On the other hand, we have $f(b_{\varphi(m)})\star g(b_{\varphi(m)})\leq \sup_{x\in G} f(x)\star g(x)$. Using Proposition \ref{util}, the continuity of $\star$ and the fact that $g(b_{\varphi(m)})\,{\xrightarrow {\textnormal{w}}}\, \mathcal{H}_0$ (when $m \longrightarrow +\infty$), we obtain that 
$$F\leq \sup_{x\in G} f(x)\star g(x).$$ Thus, $\lim_m f(b_m)=F=\sup_{x\in G} f(x)\star g(x)$. Finally, we have that $$\sup_{x\in G} f(x)\star g(x)=\lim_m f(b_m)=\lim_m(\lim_n D(a_n,b_m)).$$ Using the same arguments by replacing $f$ by $g$ we also get that 
$$\sup_{x\in G} f(x)\star g(x)=\lim_n g(a_n)=\lim_n(\lim_m D(a_n,b_m)).$$
Hence,
\begin{eqnarray*}
\sup_{x\in G} f(x)\star g(x)&=&\lim_mf(b_m)\\
                                 &=&\lim_n g(a_n)\\
                                 &=&\lim_{n,m} D(a_n,b_m).
\end{eqnarray*}
Finally, it is not difficult to see that $\lim_{n,m} D(a_n,b_m)=\lim_{n} D(a_n, b_n)$ by using the fact that $(a_n), (b_n)\subset G$ are Cauchy sequences. This finish the proof.

\end{proof}
\vskip5mm
Now, we give the proof of Theorem \ref{thmc}.

\begin{proof}[Proof of Theorem \ref{thmc}] $(1)$ Let $a, b \in G$. We have 
\begin{eqnarray*}
\D(\delta_a,\delta_b)&=&\sup_{x\in G} \delta_a(x)\star \delta_b(x) \\
                        &=& \sup_{x\in G} D(x,a)\star D(x,b)\\
                        &\leq& \sup_{x\in G} D(a,b)\\
                        &=& D(a,b).
\end{eqnarray*}
On the other hand, 
\begin{eqnarray*}
\D(\delta_a,\delta_b)&=&\sup_{x\in G} \delta_a(x)\star \delta_b(x) \\
                        &\geq& \delta_a(b)\star \delta_b(b)\\
                        &=& D(a,b)\star \mathcal{H}_0 \\ 
                        &=& D(a,b).
\end{eqnarray*}
Thus, $\D(\delta_a,\delta_b)=D(a,b)$, for all $a,b\in G$.
\vskip5mm
$(2)$ This part is given by the following two steps.
\vskip5mm
\noindent {\it Step 1:} $(\Pi(G), \D,\star)$ is a probabilistic metric space. Indeed, 

$(i)$ First, we prove that $\D(f,f)=\mathcal{H}_0$ for all $f\in \Pi(G)$. By Lemma \ref{Lelemme2} we have $\D(f,f)=\lim_{n,m} D(a_n,a_m)=\mathcal{H}_0$, where $(a_n)$ is a Cauchy sequence such that $D(a_n, x)\,{\xrightarrow {\textnormal{w}}}\, f(x)$ for all $x\in G$. Conversely, let $f, g \in \Pi(G)$ and suppose that
$\D(f,g)= \mathcal{H}_0$. There exists a Cauchy sequence $(a_n)\subset G$, such that $D(a_n, x)\,{\xrightarrow {\textnormal{w}}}\, f(x)$ for all $x\in G$. Since $g$ is a probabilistic $1$-Lipschitz map, then $D(a_n,x)\star g(x) \leq g(a_n)\leq \mathcal{H}_0$. Let $(g(a_{\varphi(n)}))$ be any (weak) convergent subsequence to some limit $L\in \Delta^+$. Since, $f(x)\star g(x)=\lim_n(D(a_n,x)\star g(x))=\lim_n(D(a_{\varphi(n)},x)\star g(x))$, we have (by Proposition \ref{util}) $$f(x)\star g(x)\leq L\leq \mathcal{H}_0.$$ 
Taking the supremum over $x\in G$, we get that 
$$\mathcal{H}_0=\D(f,g)=\sup_{x\in G} f(x)\star g(x)\leq L\leq \mathcal{H}_0.$$
In other words, $\lim_n g(a_{\varphi(n)})= \mathcal{H}_0$ for any convergent subsequence. Since $(\Delta^+,\textnormal{w})$ is a compact metrizable space, we get that $\lim_n g(a_n)=\mathcal{H}_0$. Now, since $g$ is a probabilistic $1$-Lipschitz map, then $D(a_n,x)\star g(a_n) \leq g(x)$ and so by taking the limit and using Proposition \ref{util} we obtain that $f(x)=f(x)\star \mathcal{H}_0\leq g(x)$. Using a similar arguments, we also have that $g(x)\leq f(x)$. Hence, $f=g$.  

$(ii)$ it is clear that $\D(f,g)=\D(g,f)$ for all $f, g \in \Pi(G)$.

$(iii)$ Let $f, g, h \in \Pi(G)$ and let us prove that $\D(f,g)\star \D(g,h)\leq \D(f,h)$. There exists a Cauchy sequences $(a_n), (b_n), (c_n)\subset G$ such that $D(a_n,x)\,{\xrightarrow {\textnormal{w}}}\, f(x)$, $D(b_n,x)\,{\xrightarrow {\textnormal{w}}}\, g(x)$ and $D(c_n,x)\,{\xrightarrow {\textnormal{w}}}\, h(x)$ for all $x\in G$. Then, using the properties of $D$, the continuity of $\star$ and Lemma \ref{Lelemme2}, we have

\begin{eqnarray*}
(\sup_{x\in G} f(x)\star g(x))\star (\sup_{x\in G} g(x)\star h(x))&=& \lim_{n,m}D(a_n,b_m)\star \lim_{m,p}D(b_m,c_p)\\
                                                                  &=& \lim_{n,m, p} (D(a_n,b_m)\star D(b_m,c_p))\\
                                                                  &\leq& \lim_{n, p} D(a_n,c_p)\\
                                                                  &=& \sup_{x\in G} f(x)\star h(x)
\end{eqnarray*}
Hence, $\D(f,g)\star \D(g,h)\leq \D(f,h)$.
\vskip5mm
\noindent {\it Step 2:} The probabilistic metric space $(\Pi(G),\D,\star)$ is complete. Let $(f_k)\subset \Pi(G)$ be a Cauchy sequence. For each $k\in \N$, there exists a Cauchy sequence $(a_{(n,k)})_{n\in \N} \subset G$ such that $D(a_{(n,k)}, x)\,{\xrightarrow {\textnormal{w}}}\, f_k(x)$, when $n\longrightarrow +\infty$ for all $x\in G$. This also implies that for all $k\in \N$
\begin{eqnarray}\label{cocomp1}
\D(\delta_{a_{(n,k)}},f_k)&=&f_k(a_{(n,k)})\,{\xrightarrow {\textnormal{w}}}\,\mathcal{H}_0,
\end{eqnarray}
when $n\longrightarrow +\infty$. Thus, for all $k\in \N$, there exists $N_{k}\in \N$ such that for all $n\geq N_{k}$
\begin{eqnarray}\label{cocomp7}
1-\frac{1}{k} \leq \D(\delta_{a_{(n,k)}},f_k)(\frac{1}{k})=f_k(a_{(n,k)})(\frac{1}{k})
\end{eqnarray}

Since $\D$ is a probabilistic metric, we have that for all $n,m, k,p\in \N$
\begin{eqnarray}\label{cocomp2}
\D(\delta_{a_{(n,k)}},f_k)\star \D(f_k, f_{p})\star \D(f_{p}, \delta_{a_{(m,p)}})&\leq& \D(\delta_{a_{(n,k)}},\delta_{a_{(m,p)}})\\
                                                                     &=& D(a_{(n,k)},a_{(m,p)}).\nonumber
\end{eqnarray}
Using Proposition \ref{util}, Lemma \ref{Lelemme2}, the above inequality and (\ref{cocomp1}), we get for all $k,p\in \N^*$

\begin{eqnarray*}
\D(f_k, f_{p})&\leq& \lim_{n,m}D(a_{(n,k)},a_{(m,p)}).\nonumber
\end{eqnarray*}
%Since $(f_k)$ is a Cauchy sequence, there exists $\overline{N}_\varepsilon\in \N$ such that, for all $k,p\geq \overline{N}_\varepsilon$,
%\begin{eqnarray*}
%1-\varepsilon \leq \D(f_k, f_{p})(\varepsilon)&\leq& \liminf_{n}\liminf_{m} D(a_{(n,k)},a_{(m,p)})(\varepsilon).\nonumber
%\end{eqnarray*}

We can choose a map $\varphi :\N\longrightarrow \N$ such that $\varphi(k)\geq N_{k}$ for all $k\in \N^*$ and such that for all $k,p, s \in \N^*$,
\begin{eqnarray*}
\D(f_k, f_{p})(\frac{1}{s})\leq \lim_{n,m}(D(a_{(n,k)},a_{(m,p)})(\frac{1}{s})) \leq D(a_{(\varphi(k),k)},a_{(\varphi(p),p)})(\frac{1}{s})+(\frac{1}{k}+\frac{1}{p}).\nonumber
\end{eqnarray*}
Since, elements of $\Delta^+$ are increasing, for all $t>0$, there exists an integer number $s_t\in \N$ such that $\frac{1}{s_t}\leq t$ and
\begin{eqnarray*}
\D(f_k, f_{p})(\frac{1}{s_t})-(\frac{1}{k}+\frac{1}{p})&\leq& D(a_{(\varphi(k),k)},a_{(\varphi(p),p)})(\frac{1}{s_t})\nonumber\\
                                                       &\leq& D(a_{(\varphi(k),k)},a_{(\varphi(p),p)})(t).
\end{eqnarray*}
By sending $k, p$ to $+\infty$, using the fact that $(f_k)$ is Cauchy we have for all $t>0$
\begin{eqnarray*}
1=\lim_{k,p\rightarrow +\infty} \D(f_k, f_{p})(\frac{1}{s_t})=\lim_{k,p\rightarrow +\infty} D(a_{(\varphi(k),k)},a_{(\varphi(p),p)})(t).
\end{eqnarray*}
In other words, the sequence $(a_{(\varphi(k),k)})_{k\in \N}$ is Cauchy in $G$. Using Lemma \ref{Lelemme}, there exists a Probabilistic $1$-Lipschitz map $f$ such that, for all $x\in G$ we have $D(a_{(\varphi(k),k)}, x)\,{\xrightarrow {\textnormal{w}}}\, f(x)$ and in particular $f(a_{(\varphi(k),k)})\,{\xrightarrow {\textnormal{w}}}\, \mathcal{H}_0$, when $k\longrightarrow +\infty$. This implies that $f\in\Pi(G)$ and
\begin{eqnarray}\label{cocomp5}
\D(\delta_{a_{(\varphi(k),k)}},f)=f(a_{(\varphi(k),k)}) \,{\xrightarrow {\textnormal{w}}}\,\mathcal{H}_0, \textnormal{ when } k\longrightarrow +\infty.
\end{eqnarray}
Since $\D$ is a probabilistic metric on $\Pi(G)$, we have  
\begin{eqnarray}\label{cocomp6}
 \D(f_k,\delta_{a_{(\varphi(k),k)}})\star \D(\delta_{a_{(\varphi(k),k)}},f)&\leq& \D(f_k,f)\\
                        &\leq& \mathcal{H}_0.\nonumber
\end{eqnarray}
We know that:

$(a)$ on one hand, from (\ref{cocomp7}) (since $\varphi(k)\geq N_k$ for all $k$), we have that: $\forall t>0$, $\forall k\in \N^*$ such that $\frac{1}{t}\leq k$,
 $$1-\frac{1}{k} \leq \D(\delta_{a_{(\varphi(k),k)}},f_k)(\frac{1}{k})=f_k(a_{(\varphi(k),k)})(\frac{1}{k})\leq f_k(a_{(\varphi(k),k)})(t).$$
It follows that $\lim_{k\rightarrow +\infty} f_k(a_{(\varphi(k),k)})(t)=1$ for all $t>0$. In other words, $$\D(f_k,\delta_{a_{(\varphi(k),k)}})=f_k(a_{(\varphi(k),k)})\,{\xrightarrow {\textnormal{w}}}\,\mathcal{H}_0,$$ when $k\rightarrow +\infty$.

$(b)$ on the other hand, we know from (\ref{cocomp5}) that $\D(\delta_{a_{(\varphi(k),k)}},f)\,{\xrightarrow {\textnormal{w}}}\,\mathcal{H}_0$, when $k\longrightarrow +\infty.$
\vskip5mm
Hence, using $(a)$, $(b)$ and (\ref{cocomp6}) (with Proposition \ref{util}), we get that $\D(f_k,f)\,{\xrightarrow {\textnormal{w}}}\,\mathcal{H}_0$, when $k\longrightarrow +\infty.$ This finish the proof of $(2)$.
\vskip5mm
$(3)$ Since $(\Pi(G), \D)$ is complete and $\mathcal{G}(G)\subset \Pi(G)$, clearly $\overline{\mathcal{G}(G)}^{\D}\subset \Pi(G)$. For the converse, let $f\in \Pi(G)$. Thus, there exist a Cauchy sequence $(a_n)\subset G$ such that $D(a_n, x)\,{\xrightarrow {\textnormal{w}}}\, f(x)$ for all $x\in G$, in particular $f(a_n)\,{\xrightarrow {\textnormal{w}}}\, \mathcal{H}_0$, when $n\rightarrow +\infty$ (since $(a_n)$ is a Cauchy sequence). Hence, since $f$ is $1$-Lipschitz, we have $$\D(\delta_{a_n}, f)=\sup_{x\in G} D(a_n,x)\star f(x)=f(a_n),$$
which shows that $\D(\delta_{a_n}, f)\,{\xrightarrow {\textnormal{w}}}\, \mathcal{H}_0$. In other words, $f\in \overline{\mathcal{G}(G)}^{\D}$. Thus, $\Pi(G)= \overline{\mathcal{G}(G)}^{\D}$ and so $\mathcal{G}(G)$ is dense in $(\Pi(G), \D,\star)$. 
\end{proof}
In the group framework, we have the following theorem.
\begin{theorem} \label{corc} Let $(G,\cdot,D,\star)$ be a probabilistic invariant metric group having $e$ as identity element, where $\star$ is a continuous triangle function. Then $(\Pi(G), \odot, \D,\star)$ is a probabilistic invariant complete metric group having $\delta_e$ as identity element.
\end{theorem}
\begin{proof} The fact that $(\Pi(G),\D,\star)$ is a probabilistic complete metric space is given by Theorem \ref{thmc}. Let us prove that $(\Pi(G), \odot)$ is a group having $\delta_e$ as identity element. The associative property is given by Proposition \ref{prop13}. We prove that $f\odot g \in \Pi(G)$, whenever $f, g \in\Pi(G)$. Indeed, from Lemma \ref{Lelemme}, there exists Cauchy sequences $(z_n), (s_n)\subset G$ such that $D(z_n,x) \,{\xrightarrow {\textnormal{w}}}\, f(x)$ and $D(s_n,x) \,{\xrightarrow {\textnormal{w}}}\, g(x)$ for all $x\in G$, when $n \longrightarrow +\infty$. We prove that $(z_n s_n)$ is a Cauchy sequence and that $D(z_ns_n,x)\,{\xrightarrow {\textnormal{w}}}\, (f\odot g)(x)$ for all $x\in G$, when $n \longrightarrow +\infty$. Indeed, using the invariance of $D$ we have, for all $n, p \in \N$,
$$ D(s_n,s_p)\star D(z_n,z_p)=D(z_ns_n,z_ns_p)\star D(z_ns_p,z_ps_p)\leq D(z_ns_n,z_ps_p).$$
This implies that $(z_n s_n)$ is a Cauchy sequence. Let us set $h_x(y):=g(y^{-1}x)$ for all $x, y\in G$. The sequence $(xs^{-1}_n)$ is also a Cauchy sequence and we have $D(y,xs^{-1}_n)=D(s_n,y^{-1}x) \,{\xrightarrow {\textnormal{w}}}\, g(y^{-1}x)=h_x(y)$ for all $x,y\in G$, when $n \longrightarrow +\infty$. Using Lemma \ref{Lelemme2} with $f$ and $h_x$ for each $x\in G$, we have that
\begin{eqnarray}
f\odot g (x)=\sup_{y\in G} f(y)\star h_x(y)=\lim_{n} D(z_n,xs^{-1}_n)=\lim_n D(z_ns_n,x).
\end{eqnarray}
Hence, $(z_n s_n)$ is a Cauchy sequence and $D(z_ns_n,x)\,{\xrightarrow {\textnormal{w}}}\, f\odot g(x)$ for all $x\in G$, when $n\longrightarrow +\infty $. This shows that $f\odot g \in \Pi(G)$.
\vskip5mm
Now let $f\in \Pi(G)$ and let us prove that there exists $g\in \Pi(G)$ such that $f\odot g=g\odot f= \delta_e$. From the definition of $\Pi(G)$ there exists a Cauchy sequence $(z_n)\subset G$ such that $D(z_n,x) \,{\xrightarrow {\textnormal{w}}}\, f(x)$ for all $x\in G$, in particular $f(z_n)\,{\xrightarrow {\textnormal{w}}}\, \mathcal{H}_0$, when $n \longrightarrow +\infty$. Since the probabilistic metric $D$ is invariant, then $(z_n^{-1})\subset G$ is also a Cauchy sequence. By Lemma \ref{Lelemme}, there exists a probabilistic $1$-Lipschitz map $g$ such that  $D(z^{-1}_n,x) \,{\xrightarrow {\textnormal{w}}}\, g(x)$ for all $x\in G$, when $n \longrightarrow +\infty$. In particular $g\in \Pi(G)$. We show that $f\odot g=\delta_e$. Indeed, on one hand we have for all $x,y \in G$
\begin{eqnarray*}
D(z_n,y)\star D(z_n^{-1},y^{-1}x)&=& D(e,yz^{-1}_n)\star D(yz^{-1}_n,x)\leq D(e,x)=\delta_e(x).
\end{eqnarray*}
Thus, by using Proposition \ref{util} we have that for all $x, y \in G$, $f(y)\star g(y^{-1}x)\leq \delta_e(x)$ and hence $f\odot g(x) \leq \delta_e(x)$. On the other hand, we know that $f(z_n)\,{\xrightarrow {\textnormal{w}}}\, \mathcal{H}_0$, when $n \longrightarrow +\infty$. Also we have that $g(z^{-1}_n x)\,{\xrightarrow {\textnormal{w}}}\, \delta_e(x)$, when $n \longrightarrow +\infty$. Indeed, $$\lim_n g(z^{-1}_n x)=\lim_n(\lim_p D(z^{-1}_p,z^{-1}_n x))=\lim_n(\lim_p D(z_n z^{-1}_p, x))=D(e,x)=\delta_e(x).$$
Since $f(z_n)\star g(z^{-1}_n x)\leq f\odot g(x)$, then by using Proposition \ref{util} we get that $\delta_e(x)\leq f\odot g(x)$ for all $x\in G$. Finally, we proved that $f\odot g =\delta_e$ and in a similar way we have that $g\odot f= \delta_e$.
\vskip5mm
Now, let us show that $\D$ is invariant. Indeed, let $f, g, h\in \Pi(G)$. So, there exists $(z_n), (s_n), (r_n)\subset G$ Cauchy sequences such that $D(z_n,x)\,{\xrightarrow {\textnormal{w}}}\, f(x)$, $D(s_n,x)\,{\xrightarrow {\textnormal{w}}}\, g(x)$ and $D(r_n,x)\,{\xrightarrow {\textnormal{w}}}\, h(x)$ , for all $x\in G$, when $n\longrightarrow +\infty$. We proved above that $f\odot h, g\odot h\in \Pi(G)$ and $D(z_n r_n,x) \,{\xrightarrow {\textnormal{w}}}\, f\odot h(x)$ and $D(s_n r_n,x) \,{\xrightarrow {\textnormal{w}}}\, g\odot h(x)$ for all $x \in G$, when $n\longrightarrow +\infty$. From Lemma \ref{Lelemme2} and the invariance of $D$, we have that 
\begin{eqnarray}
\D(f\odot h, g\odot h)=\lim_n D(z_nr_n,s_nr_n)=\lim_n D(z_n,s_n)=\D(f,g).
\end{eqnarray}
In a similar way we prove that $\D(h\odot f, h\odot g)=\D(f,g)$. This complete the proof. 
\end{proof}
In the following corollary, we give as consequence the completion of probabilistic metric space (a result proved by H. Sherwood in \cite{SHE1}) and also give the new result concerning the completion of probabilistic invariant metric group.
\begin{corollary} \label{corcomp} Every probabilistic metric space (resp. Every probabilistic invariant metric group) $(G,D,\star)$, where $\star$ is  a continuous triangle function, has a probabilistic metric completion (resp. has a probabilistic invariant metric group completion). More precisely, the space $(\Pi(G),\D, \star)$ (resp.  the group $(\Pi(G),\odot, \D, \star)$) is the probabilistic completion (resp. the probabilistic group completion) of $(G,D,\star)$.
\end{corollary}
\begin{proof} We know from Theorem \ref{thmc} (resp. from Theorem \ref{corc}) that $\delta: (G,D,\star) \longrightarrow (\Pi(G), \D,\star)$ is an isometric embedding (resp. an isometric homomorphism embedding). We also know that $(\Pi(G), \D,\star)$ is complete and that $\delta(G)$ is dense in $(\Pi(G), \D,\star)$. To finish the proof, it remains to prove the uniqueness of the completion which is a classical and known fact. Suppose that $(X,\Lambda,\star)$ is a probabilistic complet metric space (resp. a probabilistic invariant complete metric group) and $\varphi : G\longrightarrow X$ is an isometry (resp. an isometric group homomorphism) such that $\varphi(G)$ is dense in $X$.  Let $f\in \Pi(G)$, then there exists a Cauchy sequence $(a_n)\subset G$ such that $\D(\delta_{a_n},f)\,{\xrightarrow {\textnormal{w}}}\, \mathcal{H}_0$. Since $(a_n)\subset G$ is a Cauchy sequence in $G$, then $(\varphi(a_n))$ is a Cauchy sequence in $(X,\Lambda,\star)$ which is complet. So there exists $g\in X$ such that $\Lambda(\varphi(a_n),g)\,{\xrightarrow {\textnormal{w}}}\, \mathcal{H}_0$. Now, we define the map $T:  \Pi(G) \longrightarrow X$ by $T(f)=g$. It is a well known fact that this map is a well defined surjective isometry (resp. isometric isomorphism).
\end{proof}

%\begin{proposition} Let $f, g : G \longrightarrow (\Delta^+,\star, \leq)$ be two maps and $L, K\in \Delta^+$. Then, we have $\D_%\sigma(\langle f, L\rangle,\langle g, K\rangle)=\D(f,g)\star K\star L.$
%\end{proposition}
%
%
\section{\bf A Probabilistic Banach-Stone Type Theorem} \label{S4}
In this section, we will show that $Lip^1_\star(G,\Delta^+)$ can be provided with a probabilistic metric $\overline{\D}$, such that $(Lip^1_\star(G,\Delta^+),\odot,\overline{\D},\star)$ will be a probabilistic complete metric monoid in which the probabilistic invariant metric group $G$ embeds isimetricaly a group. Then, we caracterize the set of all invertible elements of $(Lip^1_\star(G,\Delta^+),\odot,\overline{\D},\star)$ (Theorem \ref{theorem2}) and prove a Banach-Stone type theorem (Theorem \ref{Ba-Sto}).
\subsection{The Probabilistic Complete Metric Space $(Lip^1_\star(G,\Delta^+),\overline{\D},\star)$.}
Given a probabilistic metric space $(G, D,\star)$, we build a probabilistic metric $\overline{\D}: Lip^1_\star(G,\Delta^+)\times Lip^1_\star(G,\Delta^+)\longrightarrow \Delta^+$ extending the probabilistic metric $\D$ defined on $\Pi(G)\times \Pi(G)$ as follows: 
\[\overline{\D} (f,g)=
\left\{
\begin{array}{rl}
&\D(f,g) \textnormal{ if } (f,g)\in \Pi(G)\times \Pi(G)\\
&\mathcal{H}_0 \textnormal{ if } f=g\\
&\mathcal{H}_{\infty} \textnormal{ if } f\neq g; (f,g)\in [\Pi(G)\times \Pi(G)]^C
\end{array}
\right.
\] 
where, the set $$[\Pi(G)\times \Pi(G)]^C:=(Lip^1_\star(G,\Delta^+)\times Lip^1_\star(G,\Delta^+))\setminus(\Pi(G)\times \Pi(G))$$ denotes the complement of $\Pi(G)\times \Pi(G)$ in $Lip^1_\star(G,\Delta^+)\times Lip^1_\star(G,\Delta^+)$.  
\vskip5mm
The interest of the probabilistic metric $\overline{\D}$ is that this allows to keep the distances between the points of $ G $ after embedding it in the space $Lip^1_\star(G,\Delta^+)$ via the map $\delta$. This will play an important role in Theorem \ref{Ba-Sto}.
\begin{proposition} \label{proposition.3} Let $(G,D,\star)$ be a probabilistic metric space such that $\star$ is continuous. Then, we have

$(1)$ the space $(Lip^1_\star(G,\Delta^+),\overline{\D},\star)$ is a probabilistic complete metric space in which, the probabilistic metric space $(G,D,\star)$ embeds isometrically via the map $$\delta: (G,D,\star) \longrightarrow (Lip^1_\star(G,\Delta^+), \overline{\D},\star).$$

$(2)$ if moreover $G$ is a probabilistic invariant metric group and $\star$ is sup-continuous, the map $$\delta: (G,\cdot) \longrightarrow (Lip^1_\star(G,\Delta^+), \odot)$$ is a group homomorphism
\end{proposition}
\begin{proof} It suffices to prove that $(Lip^1_\star(G,\Delta^+),\overline{\D},\star)$ is a probabilistic complete metric space. The fact that $\delta: (G,D,\star) \longrightarrow (Lip^1_\star(G,\Delta^+), \overline{\D},\star)$ is an isometry (resp. an isometric homomorphism) is given by Theorem \ref{thmc} (resp. by Theorem \ref{corc}).

\vskip5mm

{\it Step 1:} we show that $\overline{\D}$ is a probabilistic metric on $Lip^1_\star(G,\Delta^+)$. 

$(i)$ Let $f, g \in Lip^1_\star(G,\Delta^+)$. Then, $\overline{\D}(f,g)= \mathcal{H}_0$ if and only if $f=g$ since $\D$ is a probabilistic metric on $\Pi(G)$. 

$(ii)$ it is clear that $\overline{\D}(f,g)=\overline{\D}(g,f)$ for all $f, g \in Lip^1_\star(G,\Delta^+)$.

$(iii)$ let $f, g, h \in Lip^1_\star(G,\Delta^+)$ and let us prove that $\overline{\D}(f,g)\star \overline{\D}(g,h)\leq \overline{\D}(f,h)$. This is clear in the case where $(f,g), (g,h)\in \Pi(G)\times \Pi(G)$ since $\D$ is a probabilistic metric. This is also clear in the case where $(f,g), (g,h)\in [\Pi(G)\times \Pi(G)]^C$ since $\mathcal{H}_{\infty}\star \mathcal{H}_{\infty}\leq \mathcal{H}_0\star\mathcal{H}_{\infty}=\mathcal{H}_{\infty}$. Suppose now that $(f,g) \in \Pi(G)\times \Pi(G)$ and $(g,h) \in [\Pi(G)\times \Pi(G)]^C$. In this case, also $(f,h) \in [\Pi(G)\times \Pi(G)]^C$, thus
\begin{eqnarray*}
\overline{\D}(f,g)\star \overline{\D}(g,h)&=& \D(f,g)\star \mathcal{H}_{\infty}\\
                    &\leq& \mathcal{H}_0\star \mathcal{H}_{\infty}\\              
                    &=& \mathcal{H}_{\infty}=\overline{\D}(f,h).\\
\end{eqnarray*}

{\it Step 2:} We show that $(Lip^1_\star(G,\Delta^+),\overline{\D},\star)$ is complete. Indeed, let $(f_k)\subset Lip^1_\star(G,\Delta^+)$ be a Cauchy sequence. 

{\it Case 1:} Threre exists $N\in \N$ such that for all $k\geq N$, $f_k \in Lip^1_\star(G,\Delta^+)\setminus \Pi(G)$. In this case, the sequence $(f_k)$ is necessarily a stationary sequence and so converges. 

{\it Case 2:} There exists a subsequence $(f_{\varphi(k)}) \subset \Pi(G)$. In this case, since $(\Pi(G),\D,\star)$ is complete, there exists $f\in \Pi(G)$ such that $\D(f_{\varphi(k)}, f)\,{\xrightarrow {\textnormal{w}}}\, \mathcal{H}_0$. So, using the following inequality 
$$\overline{\D}(f_k, f_{\varphi(k)})\star \D(f_{\varphi(k)},f)=\overline{\D}(f_k, f_{\varphi(k)})\star \overline{\D}(f_{\varphi(k)},f)\leq \overline{\D}(f_k,f),$$
and taking the limit when $k \longrightarrow +\infty$, we get that $\overline{\D}(f_k, f)\,{\xrightarrow {\textnormal{w}}}\, \mathcal{H}_0$.
\end{proof}
%%%%%%%%%%%%%%%
%%%%%%%%%%%%%%%%
%%%%%%%%%%%%%%%%
\subsection{The Group of Units of the Monoid $(Lip^1_\star(G,\Delta^+),\odot)$.}
We are going to caracterize the set of all invertible elements of the monoid $(Lip^1_\star(G,\Delta^+),\odot)$. The main result of this section Theorem \ref{theorem2} will be used to obtain a probabilistic Banach-Stone type theorem.  
\vskip5mm
Let $\mathcal{U}(\Delta^+)$ (resp. $\mathcal{U}(Lip^1_\star(G,\Delta^+))$) denotes the group of all invertible element of $(\Delta^+, \star)$ (resp. of $(Lip^1_\star(G,\Delta^+),\odot)$). Recall that if $f\in Lip^1_\star(G,\Delta^+)$ and $F\in \Delta^+$, then $\langle f, F\rangle (x):=f(x)\star F$ for all $x\in G$. Recall also that $\mathcal{G}(G):=\lbrace \delta_a/ a\in G \rbrace$. By $\langle \Pi(G), \mathcal{U}(\Delta^+) \rangle$, we denote the following subgroup of $(Lip^1_\star(G,\Delta^+),\odot)$
$$\langle \Pi(G), \mathcal{U}(\Delta^+) \rangle:= \lbrace \langle f, U \rangle / f\in \Pi(G) \textnormal{ and } U\in \mathcal{U}(\Delta^+)\rbrace.$$
Note that if $\mathcal{U}(\Delta^+)$ is trivial, that is $\mathcal{U}(\Delta^+)=\lbrace \mathcal{H}_0\rbrace$, then $\langle \Pi(G), \mathcal{U}(\Delta^+) \rangle= \Pi(G)$. The symbol "$\cong$" means "isometrically isomorphic".

\begin{Rem} Let $\star$ be a triangle function on $\Delta^+$ such that $\mathcal{U}(\Delta^+)=\lbrace \mathcal{H}_0\rbrace$ (see Lemma \ref{lemma1}). We equip $\Delta^+$ with the following probabilistic discret metric:
\[ D(K,L)=
\left\{
\begin{array}{rl}
\mathcal{H}_0& \textnormal{ if } K=L \\
\mathcal{H}_{\infty}& \textnormal{ otherwise } 
\end{array}
\right.
\]
Then, $(\Delta^+, \star,D)\cong(Lip^1_\star(\lbrace \mathcal{H}_0 \rbrace,\Delta^+), \odot, \overline{\D})$. 
\end{Rem}
\begin{lemma} \label{leminv} Let $(G,D,\star)$ be a probabilistic metric space (not necessarily complete). Let $(z_n)\subset G$ be a Cauchy sequence and $f : G\longrightarrow \Delta^+$ be a probabilistic $1$-Lipschitz map. Then, there exists $F\in \Delta^+$ such that $f(z_n) \,{\xrightarrow {\textnormal{w}}}\, F$, when $n\longrightarrow +\infty$.
\end{lemma}
\begin{proof}
Since $(\Delta^+,\textnormal{w})$ is a compact metrizable space, it suffices to show that every (weak) convergent subsequence of $(f(z_n))$ converges to the same limit. Indeed, let $(f(z_{\varphi(n)})$ and $(f(z_{\psi(n)})$ be two convergent subsequence and let $F$ and $K$ the respective (weak) limits. Since $f$ is a probabilistic $1$-Lipschitz map, then 
$$D(z_{\varphi(n)},z_{\psi(n)})\star f(z_{\varphi(n)})\leq f(z_{\psi(n)}).$$
Using Proposition \ref{util}, and the fact that $D(z_{\varphi(n)},z_{\psi(n)})\,{\xrightarrow {\textnormal{w}}}\,\mathcal{H}_0$ we get that $F\leq L$. In a similar way, we have that $L\leq F$.
\end{proof}
\begin{theorem} \label{theorem2} Let $\star$ be a continuous and sup-continuous triangle function. Let $(G, \cdot, D, \star)$ be an invariant probabilistic metric group (not necessarily complete). Then, we have,
$$\mathcal{U}(Lip^1_\star(G,\Delta^+))=\langle \Pi(G), \mathcal{U}(\Delta^+) \rangle.$$ Moreover, 

$(1)$ for all $f\in \Pi(G)$ and $U\in (\mathcal{U}(\Delta^+),\star)$, we have $\langle f, U \rangle^{-1}= \langle f^{-1}, U^{-1} \rangle$.

$(2)$ if $\mathcal{U}(\Delta^+)=\lbrace \mathcal{H}_0 \rbrace$, then 
\begin{eqnarray*}
(\mathcal{U}(Lip^1_\star(G,\Delta^+)),\odot,\D,\star)&=&(\Pi(G),\odot,\D,\star).
\end{eqnarray*}

$(3)$ if $\mathcal{U}(\Delta^+)=\lbrace \mathcal{H}_0 \rbrace$ and $G$ is complete, then 
\begin{eqnarray*}
(\mathcal{U}(Lip^1_\star(G,\Delta^+)),\odot,\D,\star)&=&(\mathcal{G}(G),\odot,\D,\star)\\
                                                            &\cong& (G,\cdot, D,\star).
\end{eqnarray*}
\end{theorem}
\begin{proof}
We first prove that $\lbrace \langle f, U \rangle / f\in \Pi(G), U\in \mathcal{U}(\Delta^+)\rbrace \subset \mathcal{U}(Lip^1_\star(G,\Delta^+))$. Indeed, using the definition of the operation $\odot$ and the sup-continuity of $\star$ ,  it is easy to see that 
$\langle f, U \rangle \odot \langle g, V \rangle= \langle f\odot g, U\star V\rangle$ for all $f, g\in \Pi(G)$ and all $U, V\in \mathcal{U}(\Delta^+)$. Since $\Pi(G)$ is a group by Theorem \ref{corc}, $\langle f, U \rangle \odot \langle f^{-1}, U^{-1} \rangle= \langle \delta_e, \mathcal{H}_0\rangle=\delta_e$ for all $f\in \Pi(G)$ and all $U\in \mathcal{U}(\Delta^+)$. Hence, $\langle f, U \rangle^{-1}= \langle f^{-1}, U^{-1} \rangle$.
\vskip5mm

Now, we are going to prove the converse. Indeed, let $f \in \mathcal{U}(Lip^1_\star(G,\Delta^+))$. There exists $g \in \mathcal{U}(Lip^1_\star(G,\Delta^+))$ such that 
\begin{equation} \label{eq4}
f\odot g =\delta_e
\end{equation}
In particular, we have $f\odot g(e) =\delta_e(e)=\mathcal{H}_0$. Thus, we obtain

\[ 
\left\{
\begin{array}{rl}
(f\odot g)(e)(t) = 0, \textnormal{ for all } t\leq 0 \\
(f\odot g)(e)(t) = 1, \textnormal{ for all } t>0
\end{array}
\right.
\]
So, on one hand, we have that for all $z\in G$ and for all $t\leq 0$
\begin{eqnarray} \label{eq1}
(f(z)\star g(z^{-1}))(t)=0
\end{eqnarray}
On the other hand, we have that for all $t>0$
\begin{eqnarray} \label{eq2}
\sup_{z\in G}\lbrace (f(z)\star g(z^{-1}))(t)\rbrace =1
\end{eqnarray}
For all $n\in \N$, by taking $t=\frac{1}{n}$ in (\ref{eq2}), there exists $(z_n)\subset G$ such that for all $n\in \N^*$
\begin{eqnarray*} 
1-\frac{1}{n} < (f(z_n)\star g(z^{-1}_n))(\frac{1}{n}) \leq 1
\end{eqnarray*}
Since $f(z_n)\star g(z^{-1}_n)\in \Delta^+$ is increasing then, for all $t>0$ there exists $N_t\in \N$ such that for all $n\geq N_t$ we have
\begin{eqnarray} \label{eq5}
1-\frac{1}{n} < (f(z_n)\star g(z^{-1}_n))(t) \leq 1
\end{eqnarray}
From (\ref{eq1}) and (\ref{eq5}) we have that for all $t\in \R$ 
\begin{eqnarray} \label{eq6}
\lim_{n \rightarrow +\infty} (f(z_n)\star g(z^{-1}_n))(t)= \mathcal{H}_0(t)
\end{eqnarray} 
From the definition of $f\odot g$, we get that for all $x, y \in G$
\begin{eqnarray} \label{eq3} 
f(x) \star g(y^{-1}) \leq (f\odot g)(xy^{-1})
\end{eqnarray} 
Using (\ref{eq3}), we get that 
\begin{eqnarray*}  
(f(z_n) \star g(z^{-1}_p))\star (f(z_p) \star g(z^{-1}_n)) \leq (f\odot g)(z_n z^{-1}_p)\star (f\odot g)(z_p z^{-1}_n)
\end{eqnarray*}
By the commutativity of the operation $\star$, the invariance of $D$ and (\ref{eq4}), we obtain
\begin{eqnarray} \label{eq7}  
(f(z_n) \star g(z^{-1}_n))\star (f(z_p) \star g(z^{-1}_p)) &\leq& \delta_e(z_n z^{-1}_p)\star \delta_e(z_p z^{-1}_n)\nonumber\\
                                                           &=& D(z_n,z_p)\star D(z_n,z_p).
\end{eqnarray} 
Using the fact that $\mathcal{H}_0$ is the maximal element of $\Delta^+$, we get that for all $n, p \in \N^*$
\begin{eqnarray} \label{eq8} 
D(z_n,z_p)\star D(z_n,z_p) &\leq& D(z_n,z_p)\star \mathcal{H}_0 \nonumber\\
                           &=&D(z_n,z_p)\nonumber \\
                           &\leq& \mathcal{H}_0.                         
\end{eqnarray}
From (\ref{eq7}), (\ref{eq8}), we get
\begin{eqnarray} \label{eq8'} 
(f(z_n) \star g(z^{-1}_n))\star (f(z_p) \star g(z^{-1}_p)) \leq D(z_n,z_p) \leq \mathcal{H}_0.                         
\end{eqnarray}
Using the formulas (\ref{eq6}) and (\ref{eq8'}) and the continuity of $\star$, we obtain
\begin{eqnarray*}  
\lim_{n,p}D(z_n,z_p)=\mathcal{H}_0.
\end{eqnarray*}
Hence, the sequence $(z_n)\subset G$ is a Cauchy sequence and also $(z^{-1}_n)\subset G$ is a Cauchy sequence since $(G;D,\star)$ is a probabilistic invariant metric group. By Lemma \ref{Lelemme} there exists $l_f, l_g \in \Pi(G)$, $1$-Lipschitz maps such that $D(z_n,x) \,{\xrightarrow {\textnormal{w}}}\,l_f(x)$ and $D(z^{-1}_n,x) \,{\xrightarrow {\textnormal{w}}}\,l_g(x)$ for all $x\in G$, when $n\longrightarrow +\infty$. By Lemma \ref{leminv}, there exists $L_f, K_g\in \Delta^+$ such that $f(z_n)\,{\xrightarrow {\textnormal{w}}}\, L_f$ and $g(z^{-1}_n) \,{\xrightarrow {\textnormal{w}}}\,K_g$. Using the continuity of the law $\star$ in the formula (\ref{eq6}), we get that $L_f\star K_g=\mathcal{H}_0$. So, $L_f, K_g \in (\mathcal{U}(\Delta^+),\star)$ and $K_g=L_f^{-1}$. Using the definition of $f\odot g$ and the fact that $f\odot g=\delta_e$, we have for all $x\in G$
\begin{eqnarray*}
f(x)\star g (z^{-1}_n) \leq f\odot g(xz_{n}^{-1})= \delta_e(xz_{n}^{-1})=D(x,z_{n})
\end{eqnarray*}
Using Proposition \ref{util},
\begin{eqnarray*}
f(x)\star K_g \leq l_f(x)
\end{eqnarray*}
Thus,
\begin{eqnarray*}
f(x)\leq l_f(x)\star L_f.
\end{eqnarray*}
On the other hand, since $f$ is $1$-Lipschitz, then $D(x,z_{n})\star f(z_{n})\leq f(x)$ and so using Proposition \ref{util} we obtain that $l_f(x)\star L_f\leq f(x)$. Hence, $f(x)=l_f(x)\star L_f=:\langle l_f, L_f \rangle (x)$ for all $x\in G$. Thus, $f=\langle l_f, L_f \rangle$ with $l_f\in \Pi(G)$ and $L_f\in (\mathcal{U}(\Delta^+),\star)$. This finish the proof of the converse. 
\vskip5mm
The parts $(2)$ and $(3)$ are easy consequences.
\end{proof}
\subsection{Banach-Stone Type Theorem.}
We obtain the following probabilistic Banach-Stone type theorem. Note that a general framework where the group of invertible elements of $(\Delta^+,\star)$ is trivial (i.e. $\mathcal{U}(\Delta^+)=\lbrace \mathcal{H}_0 \rbrace$) is given in Lemma \ref{lemma1} below. The following theorem shows that the structure of the probabilistic invariant complete metric group $(G, \cdot, D, \star)$ such that the group of invertible elements of $(\Delta^+,\star)$ is trivial, is completely determined by the probabilistic metric monoid structure of $Lip^1_\star(G,\Delta^+)$. This is the case in particular for every invariant complete Menger groups (Corollary \ref{corollaryend}).
% (i.e. a group $(G, \cdot, D, T)$ where $T$ is a left-continuous $t$-norm .
\begin{theorem} \label{Ba-Sto} Let $\star$ be a continuous and sup-continuous triangle function such that the group of invertible elements of $(\Delta^+,\star)$ is trivial. Let $(G, \cdot, D, \star)$ and  $(G',\cdot, D',\star)$ be two probabilistic invariant complete metric groups. Then, the following assertions are equivalent.

$(1)$ $(G, \cdot, D, \star)$ and $(G',\cdot, D',\star)$ are isometrically isomorphic as groups,

$(2)$ $(Lip^1_\star(G,\Delta^+),\overline{\D}, \odot)$ and $(Lip^1_\star(G',\Delta^+),\overline{\D},\odot)$ are isometrically isomorphic as monoids.

\end{theorem}
\begin{proof} $(1) \Longrightarrow (2)$. Suppose that $\mathcal{I} : G \longrightarrow G'$ is a group isomorphism such that, for all $x, y \in G$, $D'(\mathcal{I}(x),\mathcal{I}(y))=D(x,y)$. Then, it is easy to see that the map defined by $\Phi(f):= f\circ \mathcal{I}^{-1}$ for all $f\in Lip^1_\star(G,\Delta^+)$ maps the space $Lip^1_\star(G,\Delta^+)$ into the space $Lip^1_\star(G',\Delta^+)$ and satisfies $\Phi(f\odot g)=\Phi(f)\odot \Phi(g)$ for all $f,g \in Lip^1_\star(G,\Delta^+)$. So $\Phi$ is a monoid isomorphism. Also we can easly see that $\overline{\D}(\Phi(f), \Phi(g))=\overline{\D}(f,g)$ for all $f,g \in Lip^1_\star(G,\Delta^+)$, so $\Phi$ is an isometry.

$(2)\Longrightarrow(1)$. Suppose that $\Phi : Lip^1_\star(G,\Delta^+)\longrightarrow Lip^1_\star(G',\Delta^+)$ is an isomorphism of monoids such that $\overline{\D}(\Phi(f), \Phi(g))=\overline{\D}(f,g)$ for all $f,g \in Lip^1_\star(G,\Delta^+)$. Since an isomorphism of monoids send the group of invertibles elements on the group of invertible elements, then by using Theorem \ref{theorem2}, we obtain that the restriction $\Phi_{|\langle \mathcal{G}(G), \mathcal{U}(\Delta^+) \rangle}: \langle \mathcal{G}(G), \mathcal{U}(\Delta^+) \rangle\longrightarrow \langle \mathcal{G}(G'), \mathcal{U}(\Delta^+) \rangle$ of $\Phi$ is an isometric isomorphism. Since $\mathcal{U}(\Delta^+)=\lbrace \mathcal{H}_0 \rbrace$, then $\langle \mathcal{G}(G), \mathcal{U}(\Delta^+)=\mathcal{G}(G)$ and so $\Phi_{|\mathcal{G}(G)}: \mathcal{G}(G)\longrightarrow \mathcal{G}(G') $ is an isometric isomorphism. Using the fact that $\delta : G\longrightarrow \mathcal{G}(G)$ (resp. $\delta : G'\longrightarrow \mathcal{G}(G')$) is an isometric isomorphism, we obtain that $\mathcal{I}:=\delta^{-1}\circ \Phi_{|\mathcal{G}(G)}\circ \delta : G\longrightarrow G'$ is an isomorphism and satisfies $D'(\mathcal{I}(a),\mathcal{I}(b))=D(a,b)$ for all $a, b \in G$.

\end{proof}
\vskip5mm
In the following lemma, we prove that the set of all invertible element of $(\Delta^+, \star)$ is trivial (i.e. $\mathcal{U}(\Delta^+)=\lbrace \mathcal{H}_0 \rbrace$) in the case where the law $\star=\star_T$ is defined for all $L, K\in \Delta^+$ and for all $t\in \R$ by:
\begin{eqnarray} \label{eq11}
(L\star_T K)(t):=\sup_{s+u=t} T(L(s),K(u))
\end{eqnarray}
where, $T: [0,1]\times [0,1] \longrightarrow [0,1]$ is a left-continuous $t$-norm. 
\begin{lemma} \label{lemma1} The distribution $\mathcal{H}_0$ is the only one invertible element of the monoid $(\Delta^+,\star_T)$.
\end{lemma}
\begin{proof}
Let $L, K \in \Delta^{+}$ be such that $L\star_{T} K=\mathcal{H}_0$. Then, for all $t> 0$, $$\sup_{s+u=t} T(L(s),K(u))=1.$$ In particular, for each $n\in \N^*$,
\begin{eqnarray} \label{eq10}
\sup_{s+u=\frac{1}{n}} T(L(s),K(u))=1
\end{eqnarray}
Thus, from (\ref{eq10}), there exists a sequence $(s_n)$ of real numbers such that 
\begin{eqnarray*}
1- \frac{1}{n}< T(L(s_n),K(\frac{1}{n} - s_n)).
\end{eqnarray*}
Using the monotonocity of $T$ and the fact that $T(x,y)\leq T(1,y)=y\leq 1$ for all $x, y \in [0,1]$, we obtain for all $t> 0$, there exists $N_t\in \N$ such that for all $n\geq N_t$ we have 
\begin{eqnarray*}
1- \frac{1}{n}< T(L(s_n),K(\frac{1}{n} - s_n))\leq T(L(s_n),K(t - s_n))\leq K(t - s_n)\leq 1,
\end{eqnarray*}
Since $T(0,x)=T(x,0)=0$ for all $x\in [0,1]$, we deduce form the above inequality that necessarily $0 < s_n < \frac{1}{n}$ for all $n\in \N^*$ (since, $L(s_n)$ and $K(\frac{1}{n} - s_n)$ must be different from $0$). Thus,  $s_n\searrow 0$ and for all $t> 0$, there exists $N_t\in \N$ such that for all $n\geq N_t$
\begin{eqnarray*}
1- \frac{1}{n}< K(t - s_n)\leq 1,
\end{eqnarray*}
Since $K$ is left-continuous, we get that $K(t)=1$ for all $t>0$. On the other hand, we know that $K(t)=0$ for all $t\leq 0$. It follows that $K=\mathcal{H}_0$ and so also $L=\mathcal{H}_0$. In consequence, we have $\mathcal{U}(\Delta^+)=\lbrace \mathcal{H}_0 \rbrace.$

\end{proof}
We obtain in the following corollary the result announced in the abstract. In other words, let $T$ be a left-continuous $t$-norm then, every probabilistic invariant complete Menger group $(G, \cdot, D, T)$ is completely determined by the probabilistic metric monoid structure of $(Lip^1_\star(G,\Delta^+),\D,\odot)$.
\begin{corollary} \label{corollaryend} Let $T$ be a left-continuous $t$-norm and $(G, \cdot, D, T)$, $(G',\cdot, D', T)$ be two probabilistic invariant complete Menger groups. Then, the following assertions are equivalent.

$(1)$ $(G, \cdot, D, T)$ and $(G',\cdot, D', T)$ are isometrically isomorphic as Menger groups

$(2)$ $(Lip^1_\star(G,\Delta^+), \odot, \D,\star)$ and $(Lip^1_\star(G',\Delta^+),\odot, \D,\star)$ are isometrically isomorphic as probabilistic metric monoids.
\end{corollary}
\begin{proof} The proof is a direct consequence of Theorem \ref{Ba-Sto} and Lemma \ref{lemma1}.
\end{proof}
\vskip5mm
Also, we have that $\mathcal{U}(\Delta^+)=\lbrace \mathcal{H}_0 \rbrace$ if $\star=\star_{T^*}$ is defined for all $L, K\in \Delta^+$ and for all $t\in \R$ by:
\begin{eqnarray*}
(L\star_{T^*} K)(t):=\inf_{s+u=t} T^*(L(s),K(u))
\end{eqnarray*}
where, $T: [0,1]\times [0,1] \longrightarrow [0,1]$ is a $t$-norm and $T^*(x,y):=1-T(1-x,1-y)$ for all $x, y \in [0,1]$ ($T^*$ is a dual function of $T$ called $t$-conorm).
\begin{proposition} \label{lemma2} The distribution $\mathcal{H}_0$ is the only one invertible element of the monoid $(\Delta^+,\star_{T^*})$.
\end{proposition}
\begin{proof}
Let $L, K \in \Delta^{+}$ be such that $L\star_{T^*} K=\mathcal{H}_0$. Then, for all $t\in \R$, $$\inf_{s+u=t} T^*(L(s),K(u))=\mathcal{H}_0(t).$$ 
In other words, for all $t\in \R$
$$\sup_{s+u=t} T(1-L(s),1-K(u))=1-\mathcal{H}_0(t).$$ 
In particular, for each $n\in \N^*$,
\begin{eqnarray} \label{eq101}
\sup_{s+u=\frac{1}{n}} T(1-L(s),1-K(u))=0
\end{eqnarray}
Thus, from (\ref{eq101}), for all $n\in \N^*$, 
\begin{eqnarray*}
T(1-L(\frac{1}{n}),1-K(0))=0.
\end{eqnarray*}
Since $K(0)=0$, using the fact that $T(x,1)=x$ for all $x \in [0,1]$, we obtain that $L(\frac{1}{n})=1$ for all $n\in \N^*$. Since $L$ is nondecreasing, we get that for all $t>0$, $L(t)=1$. So that $L=\mathcal{H}_0$ and so also $K=\mathcal{H}_0$.
\end{proof}
\begin{Rem} Note that $\star_{T^*}$ is not sup-continuous and therefore not used in this paper. 
\end{Rem}
%
%As consequence of the above corollary, we have that Theorem \ref{Ba-Sto} applies for every invariant complete Menger groups $(G, \cdot, D, T)$ where $T$ is a left-continuous $t$-norm.

\bibliographystyle{amsplain}

\end{document}